\documentclass[11pt]{amsart}
\usepackage[english]{babel}
\usepackage[T1]{fontenc}
\usepackage[utf8]{inputenc}
\usepackage{amsmath}
\usepackage{amssymb}
\usepackage{amsthm}
\usepackage{mathtools}
\usepackage{booktabs}
\usepackage{enumitem}
\usepackage{multirow}
\usepackage{array}
\usepackage{longtable}
\usepackage{graphicx}
\usepackage[caption=false]{subfig}
\usepackage[rgb]{xcolor} 
\usepackage{tikz}
\usepackage{adjustbox}
\usepackage{csquotes}
\usepackage[style=numeric-comp,backend=biber,sorting=nyt]{biblatex}

\usepackage{todonotes}

\usepackage[unicode]{hyperref}

\newcommand{\NN}{\mathbb{N}} 
\newcommand{\ZZ}{\mathbb{Z}} 
\newcommand{\RR}{\mathbb{R}} 
\newcommand{\CC}{\mathbb{C}} 
\newcommand{\PP}{\mathbb{P}} 
\newcommand{\tgsp}{\mathrm{T}} 
\newcommand{\arr}{\mathcal{A}} 
\newcommand{\TT}{T} 
\newcommand{\kk}{\mathcal{K}} 
\newcommand{\bb}{\mathcal{B}} 
\newcommand{\MM}{\mathcal{M}} 
\newcommand{\GG}{\mathcal{G}} 
\newcommand{\NS}{\mathcal{S}} 
\newcommand{\HH}{\mathcal{H}} 
\newcommand{\ps}[1]{\left\langle#1\right\rangle} 
\newcommand{\st}{\mid} 
\newcommand{\card}[1]{\left|#1\right|} 
\newcommand{\rest}[1]{|_{#1}} 
\newcommand{\Bl}[2]{\operatorname{Bl}_{#1}{#2}} 
\newcommand{\XX}[1]{X_{#1}}
\newcommand{\Xdelta}{\XX{\Delta}}
\newcommand{\YY}[2][\Xdelta]{Y(#1,#2)}
\newcommand{\subfan}[1]{\Delta(#1)}

\newcommand{\dt}[2][]{t^{#1}(#2)} 

\newcommand{\II}[1]{\mathcal{I}_{#1}}
\newcommand{\play}{\mathcal{C}} 

\DeclareMathOperator{\supp}{supp}
\DeclareMathOperator{\codim}{codim}
\DeclareMathOperator{\pws}{\mathcal{P}} 
\DeclareMathOperator{\rk}{rk} 
\DeclareMathOperator{\scap}{\cap} 

\renewcommand{\bar}{\overline}

\DeclareMathOperator{\des}{des}
\DeclareMathOperator{\inv}{inv}
\DeclareMathOperator{\lec}{lec}

\newcommand{\arrA}[1]{\arr_{A_{#1}}}
\newcommand{\DeltaA}[1]{\Delta_{A_{#1}}}
\newcommand{\FFA}[1]{\mathcal{F}_{A_{#1}}}
\newcommand{\XdeltaA}[1]{X_{\DeltaA{#1}}} 
\newcommand{\YYAT}[1]{\mathcal{Y}^T(A_{#1})} 
\newcommand{\YYAH}[1]{\mathcal{Y}^H(A_{#1})} 

\theoremstyle{definition}
\newtheorem{defn}{Definition}[section]

\theoremstyle{plain}
\newtheorem{thm}[defn]{Theorem}
\newtheorem{lemma}[defn]{Lemma}
\newtheorem{prop}[defn]{Proposition}

\theoremstyle{remark}
\newtheorem{rmk}[defn]{Remark}

\newtheorem{exam}[defn]{Example}

\makeatletter
\let\@@pmod\pmod
\DeclareRobustCommand{\pmod}{\@ifstar\@pmods\@@pmod}
\def\@pmods#1{\mkern5mu({\operator@font mod}\mkern 6mu#1)}
\let\@@mod\mod
\DeclareRobustCommand{\mod}{\@ifstar\@mods\@@mod}
\def\@mods#1{\mkern5mu{\operator@font mod}\mkern 6mu#1}
\makeatother

\usetikzlibrary{calc}

\numberwithin{equation}{section}

\begin{document}

\title[A basis for the cohomology of compact models]{A basis for the cohomology of compact models of toric arrangements}

\author[G. Gaiffi]{Giovanni Gaiffi}
\address{\parbox{\linewidth}{Dipartimento di Matematica, Università di Pisa\\
         Largo B.\ Pontecorvo, 5\\
         56127 Pisa, Italy}}
\email{giovanni.gaiffi@unipi.it}

\author[O. Papini]{Oscar Papini}
\address{\parbox{\linewidth}{Istituto di Scienza e Tecnologie dell'Informazione ``A.\ Faedo''\\
         Consiglio Nazionale delle Ricerche\\
         Via G.\ Moruzzi, 1\\
         56124 Pisa, Italy}}
\email{oscar.papini@isti.cnr.it}

\author[V. Siconolfi]{Viola Siconolfi}
\address{\parbox{\linewidth}{Dipartimento di Matematica, Università di Pisa\\
         Largo B.\ Pontecorvo, 5\\
         56127 Pisa, Italy}}
\email{viola.siconolfi@dm.unipi.it}
\thanks{V.S.\ supported by PRIN 2017 `Moduli and Lie Theory', Di\-par\-ti\-men\-to di Ma\-te\-ma\-ti\-ca, U\-ni\-ver\-si\-tà di Pi\-sa.}

\begin{abstract}%
In this paper we find monomial bases for the integer cohomology rings of compact wonderful models of toric arrangements. In the description of the monomials various combinatorial objects come into play: building sets, nested sets, and the fan of a suitable toric variety. We provide some examples computed via a SageMath program and then we focus on the case of the toric arrangements associated with root systems of type \(A\). Here the combinatorial description of our basis offers a geometrical point of view on the relation between some Eulerian statistics on the symmetric group.
\end{abstract}



\maketitle

\section{Introduction}
Let $\TT$ be an $n$-dimensional torus and let $X^*(\TT)$ be its group of characters; it is a lattice of rank $n$ and, by choosing a basis, we have isomorphisms $X^*(\TT)\simeq\ZZ^n$ and $\TT\simeq(\CC^*)^n$. Given an element $\chi\in X^*(\TT)$, the corresponding character on $\TT$ will be denoted by $x_\chi\colon\TT\to\CC^*$.
\begin{defn}\label{def_toric_arrangement}
A \emph{layer} in $\TT$ is a subvariety of $\TT$ of the form
\[
\kk(\Gamma,\phi)\coloneqq\{t\in\TT\mid x_\chi(t)=\phi(\chi)\text{ for all } \chi\in\Gamma\}
\]
where $\Gamma<X^*(\TT)$ is a split direct summand and $\phi\colon\Gamma\to\CC^*$ is a homomorphism. A \emph{toric arrangement} $\arr$ is a (finite) set of layers $\{\kk_1,\dotsc,\kk_r\}$ in $\TT$. A toric arrangement is called \emph{divisorial} if every layer has codimension~1.
\end{defn}
In~\cite{deconcgaiffi1} it is shown  how to construct \emph{projective wonderful models} for  the complement \(\MM(\arr)=\TT\setminus\bigcup_i \kk_i\). A projective wonderful model is a smooth projective variety containing \(\MM(\arr)\) as an open set and such that the complement of \(\MM(\arr)\) is a divisor with normal crossings and smooth irreducible components. In~\cite{deconcgaiffi2} the integer cohomology ring of these projective wonderful models was described by showing generators and relations.

In this paper we describe a basis for the integer cohomology modules. This description calls into play the relevant combinatorial objects that characterize the geometrical and topological properties of these models: the fan of a suitable toric variety, the building set associated to the arrangement and its nested sets. 

The construction of projective models of toric arrangements is a further step in a rich theory that was originated by De~Concini and Procesi in \cite{wonderful1,wonderful2}, where they studied wonderful models for the complement of a subspace arrangement, providing both a projective and a non-projective version of their construction.

In some cases the toric and subspace constructions provide the same variety. This happens for instance when we deal with root (hyperplane or toric) arrangements of type \(A\). Therefore in this case we can compare the new basis of the cohomology described in this paper with the old one coming from the subspace construction. Part of the description of these bases is similar but there are differences, that will lead us to find a bijection between two families of graphs (labeled forests) and a geometric interpretation of the equidistribution of two statistics (\emph{des} and \emph{lec}) on the symmetric group. 

Since both subspace and toric models are involved in our results, we start providing a sketch  of the history of the theory of wonderful models from both points of view.

\subsection{Some history of linear and toric wonderful models} 
The construction of wonderful models of subspace arrangements in \cite{wonderful1,wonderful2} was originally motivated by the study of Drinfeld's construction in~\cite{drinfeld} of special solutions of the Knizhnik-Zamolodchikov equations with some prescribed asymptotic behavior, then it turned out that the role of these models is crucial in several areas of mathematics. 
For instance in the case of a complexified root arrangement of type \(A_n\) (which we will deal with in Section~\ref{Sec:esempioAn} of this paper) the minimal model coincides with the moduli spaces of stable curves of genus 0 with \(n+2\) marked points. 

In the seminal papers of De~Concini and Procesi the notions of building sets and nested sets appeared for the first time in a general version. In~\cite{wonderful1} the authors showed, using a description of the cohomology rings of the projective wonderful models to give an explicit presentation of a Morgan algebra, that the mixed Hodge numbers and the rational homotopy type of the complement of a complex subspace arrangement depend only on the intersection lattice (viewed as a ranked poset). The cohomology rings of the models of subspace arrangements were also studied in~\cite{yuzvinsky,gaiffiselecta}, where some integer bases were provided, and, in the real case, in \cite{ehkr,rains}. The arrangements associated with complex reflection groups were dealt with in~\cite{hendersonwreath} from the representation theoretic point of view and in~\cite{callegarogaiffilochak} from the homotopical point of view. 
 
The connections between the geometry of these models and the Chow rings of matroids were  pointed out first in~\cite{chow} and then in~\cite{adiprasito}, where they also played a crucial role in the study of some relevant log-concavity problems. The relations with toric and tropical geometry were enlightened for instance in~\cite{feichtnersturmfels,denham,amini}.
 
The study of toric arrangements started in~\cite{looijenga} and received then a new impulse from several points of view. In~\cite{deconciniprocesivergne} and~\cite{DCP3} the role of toric arrangements as a link between partition functions and box splines is pointed out; interesting  enumerative and combinatorial aspects have been investigated via the Tutte polynomial and arithmetics matroids in \cite{mocitoricroot,mocituttetoric,dadderiomocitoric}. As for the topology of the complement \(\MM(\arr)\) of a divisorial toric arrangement, the generators of the cohomology modules over \(\CC\) where exhibited in~\cite{deconcproc} via local nonbroken circuits sets, and in the same paper the cohomology ring structure was determined in the case of totally unimodular arrangements. By a rather general approach, Dupont in~\cite{dupont} proved the rational formality of \(\MM(\arr)\). In turn, in~\cite{CDDMP}, it was shown, extending the results in~\cite{calledelu,callegaro2019erratum} and~\cite{Pagariatwo}, that the data needed in order to state the presentation of the rational cohomology ring of \(\MM(\arr)\) is fully encoded in the poset given by all the connected components of the intersections of the layers. It follows that in the divisorial case the combinatorics of this poset determines the rational homotopy of \(\MM(\arr)\).

One of the motivations for the construction of projective wonderful models of a toric arrangement \(\arr\) in~\cite{deconcgaiffi1}, in addition to the interest in their own  geometry, was that they could be an important tool to explore the generalization of the above mentioned results to the non-divisorial case.

Indeed the presentation of the cohomology ring of these models described in~\cite{deconcgaiffi2} was used in~\cite{mocipaga} to construct a Morgan differential algebra which determines the rational homotopy type of \(\MM(\arr)\). We notice that these models,  and therefore their associated Morgan algebras, depend not only on the initial combinatorial data, but also on some choices (see Section~\ref{Sec:briefdescription} for more details). In~\cite{deconcgaiffi3} a new differential graded algebra was constructed as a direct limit of the above mentioned differential Morgan algebras: it is quasi isomorphic to any of the Morgan algebras of the projective wonderful models of $\MM(\arr)$ and it has a presentation which depends only on a set of initial discrete data extracted from \(\arr\), thus proving that in the non-divisorial case the rational homotopy type of $\MM(\arr)$ depends only on these data.

As another application of the projective wonderful models of a toric arrangement, Denham and Suciu showed in~\cite{denhamsuciu} that (in the divisorial case) \(\MM(\arr)\) is both a duality space and an abelian duality space.

\subsection{Structure of this paper}
In Section~\ref{Sec:briefdescription} we will briefly recall from~\cite{deconcgaiffi1} the construction of the projective wonderful models associated with a toric arrangement. This is done in two steps: first one embeds the torus in a suitable smooth projective toric variety \(\Xdelta\) with fan \(\Delta\), then one considers the arrangement of subvarieties (in the sense of Li~\cite{li}) given by the closures of the layers of \(\arr\).  One chooses a suitable \emph{building set} \(\GG\) of subvarieties in \(\Xdelta\) and blowups them in a prescribed order to obtain the projective wonderful model \(\YY{\GG}\). The \(\GG\)-\emph{nested sets} describe the boundary of the model. The definitions of building sets and nested sets are recalled in this section. Example~\ref{ex_arrangement} provides a non trivial instance in dimension 3 of this construction, computed with the help of a SageMath program (see~\cite{dcgp-papadima}).

In Section~\ref{Sec:presentationofcohomology} we recall from~\cite{deconcgaiffi2} the presentation of the integer cohomology ring of \(\YY{\GG}\) as a quotient of a polynomial ring via generators and relations. 

Section~\ref{Sec:maintheorem} is devoted to our main result. We provide a description of a monomial \(\ZZ\)-basis of \(H^*(\YY{\GG},\ZZ)\). Every element of this basis has two factors: one is a monomial that depends essentially on a nested set \(\NS\) of \(\GG\) with certain labels (in analogy with the case of subspace models), the other one comes from the cohomology of a toric subvariety of \(\Xdelta\) associated with \(\NS\). With the help of the above mentioned SageMath program we provide a basis for the model of Example~\ref{ex_arrangement}.

Finally, we devote Section~\ref{Sec:esempioAn} to the case of a divisorial toric arrangement of type \(A_n\). We make a canonical choice of the fan \(\Delta\), i.e.\ we take the fan associated with the Coxeter chambers. Therefore \(\Xdelta\) is the toric variety of type \(A_n\) studied for instance in~\cite{procesi90,stembridge92,stanley,dolgachevlunts} and the minimal toric projective model is isomorphic to the moduli space of stable curves of genus 0 with \(n+3\) marked points, i.e.\ to the minimal projective wonderful model of the hyperplane arrangement of type \(A_{n+1}\).

This suggests to compare the new basis described in this paper with the basis coming from \cite{yuzvinsky,gaiffiselecta}. Both bases are described by labeled graphs. On one side we have forests on \(n+1\) leaves with labels on internal vertices, equipped by an additional label: a permutation in the symmetric group \(S_j\), where \(j\) is the number of trees. On the other side we have forests on \(n+2\) leaves with labels on internal vertices. We will show an explicit bijection between these two families of forests. This will also provide us with a new combinatorial proof, with a geometric interpretation, of the equidistribution of two statistics on the symmetric group: the statistic of descents \emph{des} and the statistic \emph{lec} introduced by Foata and Han in~\cite{foatahan} (both give rise to the Eulerian numbers). 

\section{Brief description of compact models}\label{Sec:briefdescription}
In this section we recall the construction of a wonderful model starting from a toric arrangement $\arr$, mainly following~\cite{deconcgaiffi1} (see also~\cite{dcgp-papadima}).

First of all, let us fix some notation that will be used throughout this paper. Given a set $A$, we will use the symbol $\scap A$ to denote the intersection of its elements, namely
\[
\scap A=\bigcap_{B\in A}B.
\]
Recall from the Introduction that $X^*(\TT)$ is the group of characters of the torus $\TT$; likewise, we denote by $X_*(\TT)$ the group of one-parameters subgroups of $\TT$. Moreover we define the vector spaces $V=X_*(\TT)\otimes_{\ZZ} \RR$ and its dual $V^*=X^*(\TT)\otimes_{\ZZ} \RR$. The usual pairing $X^*(\TT)\times X_*(\TT)\to\ZZ$ and its extension to $V^*\times V\to\RR$ will both be denoted by the symbol $\ps{\cdot,\cdot}$. Given $\Gamma<X^*(\TT)$, we define
\begin{equation}\label{eq_vgamma}
V_\Gamma\coloneqq\left\{{v}\in V\st\ps{\chi,{v}}=0\text{ for all }\chi\in \Gamma\right\}.
\end{equation}
Given a fan $\Delta$ in $V$, the corresponding toric variety will be denoted by $\Xdelta$.

We want to build a model following the techniques described by Li in~\cite{li}: in that paper, which is  inspired by \cite{wonderful1, wonderful2, macphersonproc}, the author describes the construction of a compact model starting from an arrangement of subvarieties.

\begin{defn}\label{def_simplearrsubvar}
Let $X$ be a non-singular algebraic variety. A \emph{simple arrangement of subvarieties} of $X$ is a finite set $\Lambda$ of non-singular closed connected subvarieties properly contained in $X$ such that
\begin{enumerate}
\item for every two $\Lambda_i,\Lambda_j\in\Lambda$, either $\Lambda_i\cap\Lambda_j\in\Lambda$ or $\Lambda_i\cap\Lambda_j=\emptyset$;
\item if $\Lambda_i\cap\Lambda_j\neq\emptyset$, the intersection is \emph{clean}, i.e.\ it is non-singular and for every $y\in\Lambda_i\cap\Lambda_j$ we have the following conditions on the tangent spaces:
\[
\tgsp_y(\Lambda_i\cap\Lambda_j)=\tgsp_y(\Lambda_i)\cap\tgsp_y(\Lambda_j).
\]
\end{enumerate}
\end{defn}
\begin{defn}\label{def_arrsubvar}
Let $X$ be a non-singular algebraic variety. An \emph{arrangement of subvarieties} of $X$ is a finite set $\Lambda$ of non-singular closed connected subvarieties properly contained in $X$ such that
\begin{enumerate}
\item for every two $\Lambda_i,\Lambda_j\in\Lambda$, either $\Lambda_i\cap\Lambda_j$ is a disjoint union of elements of $\Lambda$ or $\Lambda_i\cap\Lambda_j=\emptyset$;
\item if $\Lambda_i\cap\Lambda_j\neq\emptyset$, the intersection is clean.
\end{enumerate}
\end{defn}




In the toric arrangements setting, the subvarieties will be given by the intersections of the layers of the arrangement, so we introduce the combinatorial object that describe them.
\begin{defn}
The \emph{poset of layers} of a toric arrangement $\arr$ is the set $\play(\arr)$ of the connected components of the intersections of some layers of $\arr$, partially ordered by reverse inclusion.
\end{defn}
\begin{rmk} 
\begin{enumerate}
\item The whole torus belongs to $\play(\arr)$, as it can be obtained as the intersection of no layers; we define $\play_0(\arr)\coloneqq \play(\arr)\setminus \{\TT\}$.
\item The intersection of two layers $\kk(\Gamma_1,\phi_1)$ and $\kk(\Gamma_2,\phi_2)$ is the disjont union of layers of the form $\kk(\Gamma,\phi_i)$, i.e.\ they share the same $\Gamma$, namely the saturation of $\Gamma_1+\Gamma_2$.
\end{enumerate}
\end{rmk}

Given a toric arrangement $\arr$ in a torus $\TT$, we embed $\TT$ in a suitable compact toric variety. In particular we build a toric variety whose associated fan satisfies the following \emph{equal sign} condition.

\begin{defn}
Let $\Delta$ be a fan in $V$. An element $\chi\in X^*(\TT)$ has the \emph{equal sign property} with respect to $\Delta$ if, for every cone $C\in\Delta$, either $\ps{\chi,c}\geq 0$ for all $c\in C$ or $\ps{\chi,c}\leq 0$ for all $c\in C$.
\end{defn}
\begin{defn}
Let $\Delta$ be a fan in $V$ and let $\kk(\Gamma,\phi)$ be a layer. A $\ZZ$-basis $(\chi_1,\dotsc,\chi_m)$ for $\Gamma$ is an \emph{equal sign basis} with respect to $\Delta$ if $\chi_i$ has the equal sign property for all $i=1,\dotsc,m$.
\end{defn}

We say that a toric variety $\Xdelta$ is \emph{good} for an arrangement $\arr$ if each layer of $\play(\arr)$ has an equal sign basis with respect to the fan $\Delta$. In fact in this situation the following Theorem holds; we present the statement from~\cite{deconcgaiffi2}, which summarizes Proposition~3.1 and Theorem~3.1 from~\cite{deconcgaiffi1}.

\begin{thm}[{\cite[Theorem 5.1]{deconcgaiffi2}}]
\label{thm_layer_chiusura}
For any layer $\kk(\Gamma,\phi)\in\play(\arr)$ let $H=H(\Gamma)\coloneqq\cap_{\chi\in \Gamma}\ker(x_\chi)$ be the corresponding homogeneous subtorus and let $V_\Gamma$ as in~\eqref{eq_vgamma}, i.e.
\[
V_\Gamma\coloneqq\left\{{v}\in V\st\ps{\chi,{v}}=0\text{ for all }\chi\in \Gamma\right\}.\]
\begin{enumerate}
\item \label{thm_layer_chiusura_i}
For every cone $C\in\Delta$, its relative interior is either entirely contained in $V_{\Gamma}$ or disjoint from $V_{\Gamma}$.
\item \label{thm_layer_chiusura_ii}
The collection of cones $C\in\Delta$ which are contained in $V_{\Gamma}$ is a smooth fan $\Delta_{H}$.
\item \label{thm_layer_chiusura_iii}
 $\bar{\kk(\Gamma,\phi)}$ is a smooth $H$-variety whose fan is $\Delta_{H}$.
\item \label{thm_layer_chiusura_iv}
 Let $\mathcal{O}$ be an orbit of $\TT$ in $\Xdelta$ and let $C_{\mathcal{O}}\in\Delta$ be the corresponding cone. Then
\begin{enumerate}
\item if $C_{\mathcal{O}}$ is not contained in $V_{\Gamma}$, $\bar{\mathcal{O}}\cap\bar{\kk(\Gamma,\phi)}=\emptyset$;
\item If $C_{\mathcal{O}}\subset V_{\Gamma}$, $\mathcal{O}\cap\bar{\kk(\Gamma,\phi)}$ is the $H$-orbit in $\bar{\kk(\Gamma,\phi)}$ corresponding to $C_{\mathcal{O}}\in\Delta_H$.
\end{enumerate}
\end{enumerate}
\end{thm}

As a consequence the set of the connected components of the intersections of the closures of the layers $\kk(\Gamma,\phi)\in \arr$ in $\Xdelta$ is an arrangement of subvarieties according to Li's definition.

Following \cite{deconcgaiffi1} we now introduce the wonderful model associated with an arrangement $\Lambda$ of subvarieties in a generic non-singular algebraic variety $X$. 
To do so we need to define the notion of \emph{building sets} and \emph{nested sets}.

\begin{defn}\label{def_building_set_simple}
Let $\Lambda$ be a simple arrangement of subvarieties. A subset $\GG\subseteq\Lambda$ is a \emph{building set} for $\Lambda$ if for every $L\in\Lambda\setminus\GG$ the minimal elements (w.r.t.\ the inclusion) of the set $\{G\in\GG\st L\subset G\}$ intersect transversally and their intersection is $L$. These minimal elements are called the \emph{$\GG$-factors} of $L$.
\end{defn}
\begin{defn}\label{def_nested_set_simple}
Let $\GG$ be a building set for a simple arrangement $\Lambda$. A subset $\NS\subseteq\GG$ is called ($\GG$-)\emph{nested} if for any antichain\footnote{An \emph{antichain} in a poset is a set of pairwise non-comparable elements.} $\{A_1,\dotsc,A_k\}\subseteq\NS$, with $k\geq 2$, there is an element in $\Lambda$ of which $A_1,\dotsc,A_k$ are the $\GG$-factors.
\end{defn}
\begin{rmk}
Since the empty set has no antichains of cardinality at least $2$, the definition above applies vacuously for it.
\end{rmk}
\begin{rmk}\label{rmk:empty_intersection}
We notice that if $\mathcal{H}$ is a subset of $\GG$ whose elements have empty intersection, then it cannot be contained in any $\GG$-nested set.
\end{rmk}

In case the arrangement $\Lambda$ is not simple, the definitions above apply locally: first of all, we define the \emph{restriction} of an arrangement of subvarieties $\Lambda$ to an open set $U\subseteq X$ to be the set
\[
\Lambda\rest{U}\coloneqq\{\Lambda_i\cap U\mid \Lambda_i\in\Lambda,\ \Lambda_i\cap U\neq\emptyset\}.
\]
\begin{defn}\label{def_building_set}
Let $\Lambda$ be an arrangement of subvarieties of $X$. A subset $\GG\subseteq\Lambda$ is a \emph{building set} for $\Lambda$ if there is a cover $\mathcal{U}$ of open sets  of $X$ such that
\begin{enumerate}
\item for every $U\in\mathcal{U}$, the restriction $\Lambda\rest{U}$ is simple;
\item for every $U\in\mathcal{U}$, $\GG\rest{U}$ is a building set for $\Lambda\rest{U}$.
\end{enumerate}
\end{defn}
\begin{defn}\label{def_nested_set}
Let $\GG$ be a building set for an arrangement $\Lambda$. A subset $\NS\subseteq\GG$ is called ($\GG$-)\emph{nested} if there is an open cover $\mathcal{U}$ of $X$ such that, for every $U\in\mathcal{U}$, $\Lambda\rest{U}$ is simple, $\GG\rest{U}$ is building for $\Lambda\rest{U}$ and for at least one $W\in\mathcal{U}$, $\NS\rest{W}$ is $\GG\rest{W}$-nested. (In particular $A\cap W\neq\emptyset$ for all $A\in\NS$.)
\end{defn}

Instead of defining a building set in terms of a given arrangement, it is often convenient to study the notion of ``building'' as an intrinsic property of a set of subvarieties.
\begin{defn}\label{def_absolute_building}
A finite set $\GG$ of connected subvarieties of $X$ is called a \emph{building set} if the set of the connected components of all the possible intersections of collections of subvarieties from $\GG$ is an arrangement of subvarieties, called the arrangement \emph{induced} by $\GG$ and denoted by $\Lambda(\GG)$, and $\GG$ is a building set for $\Lambda(\GG)$ according to Definition~\ref{def_building_set}.
\end{defn}
From now on, Definition~\ref{def_absolute_building} applies when we refer to a set of subvarieties as ``building'' without specifying the arrangement.

Given an arrangement $\Lambda$ of a non-singular variety $X$ and a building set $\GG$ for $\Lambda$, a wonderful model $\YY[X]{\GG}$ can be obtained as the closure of the locally closed embedding:
\[
\left(X\setminus\bigcup_{\Lambda_i\in \Lambda}\Lambda_i\right) \longrightarrow \prod_{G\in \GG}\Bl{G}{X}
\]
where $\Bl{G}{X}$ is the blowup of $X$ along $G$. 
Concretely we can build $\YY[X]{\GG}$ one step at a time, through a series of blowups, as described in the following theorem.

\begin{thm}[{see~\cite[Theorem~1.3]{li}}]\label{thm:iterblowup}
Let $\GG$ be a building set in a non-singular variety $X$. Let us order the elements $G_1,\dotsc,G_m$ of $\GG$ in such a way that for every $1\leq k \leq m$ the set $\GG_k\coloneqq\{G_1,\dotsc,G_k\}$ is building. Then if we set $X_0\coloneqq X$ and $X_k\coloneqq \YY[X]{\GG_k}$ for $1\leq k\leq m$, we have
\[
X_k=\Bl{\dt{G_k}}{X_{k-1}},
\]
where $\dt{G_k}$ denotes the dominant transform\footnote{In the blowup of a variety $M$ along a centre $F$ the dominant transform of a subvariety $Z$ coincides with the proper transform if $Z\nsubseteq F$ (and therefore it is isomorphic to the blowup of $Z$ along $Z\cap F$), and with $\pi^{-1}(Z)$ if $Z\subseteq F$, where $\pi\colon \Bl{F}{M}\to M$ is the projection. We will use the same notation $\dt{Z}$ for both the proper and the dominant transform of $Z$, if no confusion arises.} of $G_k$ in $X_{k-1}$.
\end{thm}
\begin{rmk}
Any total ordering of the elements of a building set $\GG=\{G_1,\dotsc,G_m\}$ which refines the ordering by inclusion, that is $i<j$ if $G_i\subset G_j$, satisfies the condition of Theorem~\ref{thm:iterblowup}.
\end{rmk}
Let us denote by $\pi\colon\YY[X]{\GG}\to X$ the blowup map. The boundary of $\YY[X]{\GG}$ admits a description in terms of $\GG$-nested sets.
\begin{thm}[{see~\cite[Theorem~1.2]{li}}]\label{thm:Yboundary}
The complement in $\YY[X]{\GG}$ of $\pi^{-1}(X\setminus\bigcup\Lambda_i)$ is the union of the divisors $\dt{G}$, where $G$ ranges among the elements of $\GG$. Let $\mathcal{U}$ be an open cover of $X$ such that for every $U\in\mathcal{U}$ $\Lambda\rest{U}$ is simple and $\GG\rest{U}$ is building for $\Lambda\rest{U}$. Then, given $U\in\mathcal{U}$ and $A_1,\dotsc,A_k\in\GG$, the intersection $\dt{A_1}\cap\dotsb\cap\dt{A_k}\cap\pi^{-1}(U)$ is non-empty if and only if $\{A_1\rest{U},\dotsc,A_k\rest{U}\}$ is $\GG\rest{U}$-nested; moreover, if the intersection is non-empty then it is transversal.
\end{thm}

\begin{rmk}
We notice that Definition~\ref{def_nested_set} and the statement of Theorem~\ref{thm:Yboundary} are slightly different from the ones in the literature (see \cite{li,deconcgaiffi1,deconcgaiffi2,deconcgaiffi3,dcgp-papadima}), where a subset $\NS\subseteq\GG$ is considered $\GG$-nested if $\NS\rest{U}$ is $\GG\rest{U}$-nested for every $U\in\mathcal{U}$. We think that our Definition~\ref{def_nested_set} and Theorem~\ref{thm:Yboundary} are more precise and remove an ambiguity, since they point out that the intersection property depends on the property of being nested locally in the chart $\pi^{-1}(U)$ of $\YY[X]{\GG}$.
\end{rmk}

\begin{exam}\label{ex_arrangement}
In order to compute some non-trivial examples, a series of scripts, extending the ones described in~\cite{papini}, were developed in the \texttt{SageMath} environment~\cite{sagemath}.

Let $\arr=\{\kk_1,\kk_2,\kk_3\}$ be the arrangement in $(\CC^*)^3$, with coordinates $(x,y,z)$, whose layers are defined by the equations
\begin{align*}
\kk_1\colon & xz^2=1, \\
\kk_2\colon & z=xy, \\
\kk_3\colon & xy^2=1,
\end{align*}
We can view them as $\kk_i=\kk(\Gamma_i,\phi_i)$ where $\Gamma_1<\ZZ^3$ is generated by $(1,0,2)$, $\Gamma_2$ by $(1,1,-1)$ and $\Gamma_3$ by $(1,2,0)$, and $\phi_i$ is the constant function equal to $1$ for $i=1,2,3$. Figure~\ref{fig_ex_arrangement_poset} represents the Hasse diagram of the poset of layers $\play(\arr)$.

In order to define a projective model for the arrangement, the first ingredient is a good toric variety. For this example, following the algorithm of~\cite{deconcgaiffi1}, we built a toric variety $\Xdelta$ whose fan has 72~rays and 140~maximal cones, listed in Appendix~\ref{app:examplefan} respectively. This is not the smallest fan associated with a good toric variety for this arrangement, but it has some addictional properties that are useful for the computation of a presentation for the cohomology ring of $\YY{\GG}$.

The next choice is a building set $\GG$; for this example we use the subset of the elements of $\play(\arr)$ that are pictured in a double circle in Figure~\ref{fig_ex_arrangement_poset} and obtain the model $Y_\arr=\YY{\GG}$. We will study this model in later examples of this paper.
\begin{figure}[!htb]
\centering
\begin{tikzpicture}[every node/.style={circle,fill=white,inner sep=1pt,minimum size=5mm,font=\scriptsize}]
\def\vunit{1.2}
\def\hunit{0.8}
\node[draw] (t) at (0,0) {$\TT$};
\node[draw,double] (k1) at (-\hunit,\vunit) {$\kk_1$};
\node[draw,double] (k2) at (0,\vunit) {$\kk_2$};
\node[draw,double] (k3) at (\hunit,\vunit) {$\kk_3$};
\node[draw] (l1) at (-1.5*\hunit,2*\vunit) {$L_1$};
\node[draw,double] (l2) at (-0.5*\hunit,2*\vunit) {$L_2$};
\node[draw,double] (l3) at (0.5*\hunit,2*\vunit) {$L_3$};
\node[draw] (l4) at (1.5*\hunit,2*\vunit) {$L_4$};
\node[draw,double] (p1) at (-1.5*\hunit,3*\vunit) {$P_1$};
\node[draw,double] (p2) at (-0.5*\hunit,3*\vunit) {$P_2$};
\node[draw,double] (p3) at (0.5*\hunit,3*\vunit) {$P_3$};
\node[draw,double] (p4) at(1.5*\hunit,3*\vunit) {$P_4$};
\draw[->] (t) -- (k1);
\draw[->] (t) -- (k2);
\draw[->] (t) -- (k3);
\draw[->] (k1) -- (l1);
\draw[->] (k1) -- (l2);
\draw[->] (k1) -- (l3);
\draw[->] (k2) -- (l1);
\draw[->] (k2) -- (l4);
\draw[->] (k3) -- (l2);
\draw[->] (k3) -- (l3);
\draw[->] (k3) -- (l4);
\draw[->] (l1) -- (p1);
\draw[->] (l1) -- (p2);
\draw[->] (l1) -- (p3);
\draw[->] (l1) -- (p4);
\draw[->] (l2) -- (p1);
\draw[->] (l2) -- (p3);
\draw[->] (l3) -- (p2);
\draw[->] (l3) -- (p4);
\draw[->] (l4) -- (p1);
\draw[->] (l4) -- (p2);
\draw[->] (l4) -- (p3);
\draw[->] (l4) -- (p4);
\end{tikzpicture}
\caption{Poset of layers for the arrangement of Example~\ref{ex_arrangement}, with the elements of the building set highlighted with a double circle.}
\label{fig_ex_arrangement_poset}
\end{figure}
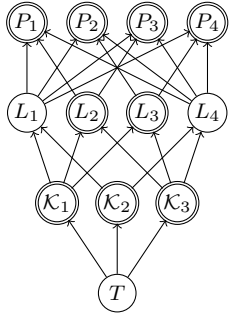
\end{exam}

\section{Presentation of the cohomology ring}\label{Sec:presentationofcohomology}
In this section we recall a presentation of the cohomology ring of the model $\YY{\GG}$. As we have seen, given a toric arrangement $\arr$ and a toric variety $\Xdelta$ which is good for it, the set $\Lambda=\{\bar{\kk}\st \kk\in\play_{0}(\arr)\}$ is an arrangement of subvarieties of $\Xdelta$ according to Li.


The cohomology ring is described as a quotient of a polynomial ring with coefficients in $H^*(\Xdelta,\ZZ)$. We present here a result by Danilov that provides an explicit presentation of the cohomology ring of a toric variety.

\begin{thm}[{\cite[Theorem~10.8]{danilov}}]\label{thm:toric_variety_cohomology}
Let $X_\Delta$ be a smooth complete $\TT$-variety. Let $\mathcal{R}$ be the set of primitive generators of the rays of $\Delta$ and define a polynomial indeterminate $C_r$ for each $r\in\mathcal{R}$. Then
\[
H^*(\Xdelta,\ZZ)\simeq\ZZ[C_r\mid r\in\mathcal{R}]/(\II{\mathrm{SR}}+\II{\mathrm{L}})
\]
where
\begin{itemize}
\item $\II{\mathrm{SR}}$ is the \emph{Stanley-Reisner ideal}
\[
\II{\mathrm{SR}}\coloneqq(C_{r_1}\dotsb C_{r_k}\mid r_1,\dotsc,r_k\text{ do not belong to a cone of }\Delta);
\]
\item $\II{\mathrm{L}}$ is the \emph{linear equivalence ideal}
\[
\II{\mathrm{L}}\coloneqq\left(\sum_{r\in\mathcal{R}}\ps{\beta,r}C_r \,\middle|\, \beta\in X^*(\TT)\right).
\]
\end{itemize}
Notice that for $\II{\mathrm{SR}}$ it is sufficient to take only the square-free monomials, and for $\II{\mathrm{L}}$ it is sufficient to take only the $\beta$'s belonging to a basis of $X^*(\TT)$.
\end{thm}
Furthermore the residue class of $C_r$ in $H^2(\Xdelta,\ZZ)$ is the cohomology class of the divisor $D_r$ associated with the ray $r$ for each $r\in \mathcal R$. By abuse of notation we are going to denote this residue class in $H^2(\Xdelta,\ZZ)$ also by $C_r$.
\begin{rmk}
Given a layer $\kk(\Gamma,\phi)$, the inclusion $j\colon\bar{\kk(\Gamma,\phi)}\hookrightarrow\Xdelta$ induces a restriction map in cohomology
\begin{equation}\label{eq_restriction_map}
j^*\colon H^*(\Xdelta,\ZZ)\to H^*(\bar{\kk(\Gamma,\phi)},\ZZ).
\end{equation}
As noted in~\cite[Proposition~5.4]{deconcgaiffi2} this map is surjective and its kernel is generated by $\{C_r\mid r\in\mathcal{R},\ r\notin V_\Gamma\}$. In the sequel, we identify $\bar{\kk(\Gamma,\phi)}$ with $\XX{\Delta_H}$, where $\Delta_H$ is the same fan of Theorem~\ref{thm_layer_chiusura} (point~\ref{thm_layer_chiusura_iii}).
\end{rmk}

\begin{exam}[Example~\ref{ex_arrangement}, continued]\label{ex_arrangement_toric_cohom}
We computed the presentation of $H^*(\Xdelta,\ZZ)$ as in Theorem~\ref{thm:toric_variety_cohomology} for the toric variety $\Xdelta$ of Example~\ref{ex_arrangement}. The cohomology ring is isomorphic to a quotient of the ring $\ZZ[C_1,\dotsc,C_{72}]$ where each indeterminate $C_i$ corresponds to the ray $r_i$ as listed in the table in Appendix~\ref{app:examplefan}. We won't report the full presentation here and give only the Betti numbers:
\begin{align*}
\rk(H^0(\Xdelta,\ZZ))&{}=1, \\
\rk(H^2(\Xdelta,\ZZ))&{}=69, \\
\rk(H^4(\Xdelta,\ZZ))&{}=69, \\
\rk(H^6(\Xdelta,\ZZ))&{}=1.
\end{align*}
\end{exam}

The presentation of the cohomology ring of $\YY{\GG}$ has been computed with an additional hypothesis on the building set.

\begin{defn}\label{def_well_connected}
A building set $\GG$ for $\Lambda$ is \emph{well-connected} if for any subset $\HH\subseteq \GG$, if the intersection $\scap \HH$ has two or more connected components, then each of these components belongs to $\GG$.
\end{defn}

\begin{exam}[Example~\ref{ex_arrangement}, continued]
The building set $\GG$ described in Example~\ref{ex_arrangement} is a well-connected building set.
\end{exam}

Some general properties of well-connected building sets are studied and presented in~\cite[Section~6]{dcgp-papadima}.

We recall here the main ingredients for the presentation of $H^*(\YY{\GG},\ZZ)$. Let $Z$ be an indeterminate and let $R=H^*(\Xdelta,\ZZ)$ viewed as $\ZZ[C_r\mid r\in\mathcal{R}]/(\II{\mathrm{SR}}+\II{\mathrm{L}})$ as in Theorem~\ref{thm:toric_variety_cohomology}. (For brevity we will use again the symbols $C_r$ instead of the corresponding equivalence classes in the quotient.) For every $G\in\Lambda$ we denote by $\Gamma_G$ the lattice such that $G=\bar{\kk(\Gamma_G,\phi)}$.

Given a pair $(M,G)\in(\Lambda\cup\{\Xdelta\})\times\Lambda$ with $G\subseteq M$, we can choose a basis $(\chi_1,\dotsc,\chi_s)$ for $\Gamma_G$ such that it is equal sign with respect to $\Delta$ and that $(\chi_1,\dotsc,\chi_k)$, with $k\leq s$, is a basis for $\Gamma_M$ (if $M$ is the whole variety $\Xdelta$, then we choose any equal sign basis of $\Gamma_G$ and let $k=0$). We define the polynomials $P^M_G\in R[Z]$ as
\[
P^M_G(Z)\coloneqq\prod_{j=k+1}^{s}\left(Z-\sum_{r\in\mathcal{R}}\min(0,\ps{\chi_j,r})C_r\right).
\]
If $G=M$, we set $P_G^G\coloneqq 1$ since it is an empty product. As shown in~\cite[Proposition~6.3]{deconcgaiffi2}, different choices of the $P_G^M$'s are possible; for example, in~\cite[Remark~4.4]{mocipaga}, the authors suggest the polynomials
\[
P^M_G(Z)=Z^{s-k}+\prod_{j=k+1}^{s}\left(-\sum_{r\in\mathcal{R}}\min(0,\ps{\chi_j,r})C_r\right).
\]

Now we define the set
\[
\mathcal{W}\coloneqq\{(G,\HH)\in\GG\times\pws(\GG)\mid G\subsetneq H\text{ for all } H\in \HH\},
\]
where $\pws(\GG)$ is the power set of $\GG$. Notice that $(G,\emptyset)\in\mathcal{W}$ for all $G\in\GG$. For each pair $(G,\HH)\in\mathcal{W}$ we define a relation $F(G,\HH)$ in the following way: let $M$ be the unique connected component of $\scap{\HH}$ that contains $G$ (as usual, if $\HH=\emptyset$ then $M=\Xdelta$), and for $G\in\GG$ let $\GG_G\coloneqq\{H\in\GG\st H\subseteq G\}$; with this information we define the polynomial $F(G,\HH)\in R[T_G\mid G\in\GG]$ as
\[
F(G,\HH)\coloneqq P_G^M\Big(\sum_{H\in \GG_G}-T_H\Big)\prod_{K\in \HH}T_K.
\]
Finally let $\mathcal{W}_0\coloneqq\{\HH\in\pws(\GG)\mid \scap\HH=\emptyset\}$. For each $\HH\in\mathcal{W}_0$ we define the polynomial $F(\HH)\in R[T_G\mid G\in\GG]$ as
\[
F(\HH)\coloneqq\prod_{K\in \HH}T_K.
\]
\begin{thm}[{\cite[Theorem~7.1]{deconcgaiffi2}}]\label{thm_relcoom}
Let $\arr$, $\Xdelta$ and $\Lambda$ be as in the beginning of this section and let $\GG$ be a well-connected building set for $\Lambda$. Let also $R=H^*(\Xdelta,\ZZ)$ viewed as a polynomial ring as in Theorem~\ref{thm:toric_variety_cohomology}. The cohomology ring of the wonderful model $H^*(\YY{\GG},\ZZ)$ is isomorphic to the quotient of $R[T_G\mid G\in\GG]$ by the ideal $\II{\GG}$ generated by
\begin{enumerate}
\item the products $C_rT_G$, with $G\in\GG$ and $r\in\mathcal{R}$ such that $r$ does \emph{not} belong to $V_{\Gamma_G}$;
\item the polynomials $F(G,\HH)$ for every pair $(G,\HH)\in\mathcal{W}$;
\item the polynomials $F(\HH)$ for every $\HH\in\mathcal{W}_0$.
\end{enumerate}
The isomorphism is given by sending \(T_G\) for $G\in \GG$ to the cohomology class associated with the divisor in the boundary which is the transform of $G$ ($\dt{G}$). Putting all together, we have
\begin{align*}
H^*(\YY{\GG},\ZZ)&{}\simeq R[T_G\mid G\in\GG]/\II{\GG}\\
&{}\simeq\ZZ[C_r,T_G\st r\in\mathcal{R},G\in\GG]/(\II{\mathrm{SR}}+\II{\mathrm{L}}+\II{\GG}).
\end{align*}
\end{thm}
\begin{rmk}
It was already noted in~\cite[Theorem 9.1]{deconcgaiffi1} that the cohomology of the projective wonderful model $\YY{\GG}$ is a free $\ZZ$-module and $H^i(\YY{\GG},\ZZ)=0$ for $i$ odd. 
\end{rmk}

\begin{exam}[Example~\ref{ex_arrangement}, continued]\label{ex_arrangement_model_cohom}
We computed the presentation of $H^*(Y_\arr,\ZZ)$ as in Theorem~\ref{thm_relcoom} for the model $Y_\arr$ of Example~\ref{ex_arrangement}. The cohomology ring is isomorphic to a quotient of the ring $\ZZ[C_1,\dotsc,C_{72},T_1,\dotsc,T_9]$, where the $C_i$'s are the same as Example~\ref{ex_arrangement_toric_cohom} and each $T_j$ corresponds to an element of $\GG$ in the following way:
\begin{align*}
\kk_1 & {}\rightsquigarrow T_1, & \kk_2 & {}\rightsquigarrow T_2, & \kk_3 & {}\rightsquigarrow T_3, & L_2 & {}\rightsquigarrow T_4, & L_3 & {}\rightsquigarrow T_5, \\ P_1 & {}\rightsquigarrow T_6, & P_3 & {}\rightsquigarrow T_7, & P_2 & {}\rightsquigarrow T_8, & P_4 & {}\rightsquigarrow T_9.
\end{align*}
Once again we won't report the full presentation here and give only the Betti numbers:
\begin{align*}
\rk(H^0(Y_\arr,\ZZ))&{}=1, \\
\rk(H^2(Y_\arr,\ZZ))&{}=75, \\
\rk(H^4(Y_\arr,\ZZ))&{}=75, \\
\rk(H^6(Y_\arr,\ZZ))&{}=1.
\end{align*}
\end{exam}

In the next Section we are going to give a description of a monomial basis of $R[T_G\mid G\in\GG]/\II{\GG}$.

\section{Main theorem}\label{Sec:maintheorem}
Let $\arr$ be a toric arrangement, let $\Xdelta$ be a good toric variety for it with associated fan $\Delta$, let $\Lambda$ be the poset of intersections of the closures of the layers of $\arr$ in $\Xdelta$ which is an arrangement of subvarieties, and let $\GG$ be a well-connected building set for $\Lambda$. 
In this Section we show a basis of the cohomology of $\YY{\GG}$ in terms of \emph{admissible functions} (see Definition~\ref{def_admissible}). The notion of admissible function is analogue to the one used in the linear setting case of subspace arrangements, which is introduced in~\cite{gaiffiselecta}.

We begin by proving a characterization of $\GG$-nested sets in this case, which will be useful in the proof of our main theorem.
\begin{prop}\label{prop_nested_characterization}
Let $\GG$ be a well-connected building set. A subset $\NS\subseteq\GG$ is $\GG$-nested if and only if for any antichain $\HH\subseteq\NS$ with at least two elements the intersection $\scap \HH$ is non-empty, connected, transversal, and does not belong to $\GG$.
\end{prop}
\begin{proof}[Proof of Proposition~\ref{prop_nested_characterization}]
\noindent\framebox{$\Rightarrow$}
Let $\NS$ be a $\GG$-nested set an let $\HH=\{A_1,\ldots,A_{k}\}$ be an antichain with $k\geq2$. The intersection $\scap\HH$ is not empty by Definition~\ref{def_nested_set} and Remark~\ref{rmk:empty_intersection}. If $\scap\HH$ is not connected then it would be the disjoint union of at least two elements in $\GG$ by well-connectedness. Let us consider $W\in\mathcal{U}$ as in Definition~\ref{def_nested_set}, so that $\NS\rest{W}$ is $\GG\rest{W}$-nested and for every $A\in\NS$ $A\cap W\neq\emptyset$, therefore $\HH\rest{W}=\{A_1\cap W,\dotsc,A_k\cap W\}$ is an antichain of $\NS\rest{W}$.
This implies that 
\[
(A_1\cap W)\cap\dotsb\cap (A_k\cap W)
\]
would be the $\GG\rest{W}$-decomposition of $G\cap W$, where $G\in\GG$ is one of the connected components of $\scap\HH$. We reached a contradiction because $G\cap W\in\GG\rest{W}$, and we deduce that $\scap\HH$ is connected and furthermore that it is not an element of $\GG$. The transversality of the intersection in $W$ also implies that $\scap\HH$ is transversal.

\noindent\framebox{$\Leftarrow$}
We now suppose that for every antichain $\HH\subseteq\NS$ with at least two elements the intersection $\scap\HH$ is non-empty, connected, transversal and does not belong to $\GG$.

Let $\mathcal{U}$ be an open cover as in Definition~\ref{def_building_set}. Let us first notice that $\scap \NS$ is equal to the intersection of the antichain given by the minimal elements of $\NS$, therefore it is non-empty and connected. Let $U\in\mathcal{U}$ such that $(\scap \NS)\cap U$ is not empty. We will prove that $\NS\rest{U}$ is $\GG\rest{U}$-nested.

Consider an antichain $\{A_1\cap U,\dotsc,A_k\cap U\}$ in $\NS\rest{U}$ with $k\geq 2$ (in particular $A_i\cap U$ is non-empty for every $i$): we will prove that $A_1\cap U,\dotsc,A_k\cap U$ are the $\GG\rest{U}$-factors of their intersection.

Let us put $\HH=\{A_1,\dotsc,A_k\}$. We observe that $\HH$ is an antichain, therefore by hypothesis we know that $A=\scap\HH$ is non-empty, connected, transversal and does not belong to $\GG$. Since $\GG\rest{U}$ is building, we know that $A\cap U$ is the transversal intersection of the minimal elements of $\GG\rest{U}$ among the ones containing $A\cap U$. We let those minimal elements be $B_1\cap U,\ldots,B_r\cap U$. For simplicity, in the rest of the proof we omit the reference to $U$. 

By minimality each of the $A_i$'s contains some of the $B_j$'s: in the next paragraph we partition $\HH$ in subsets according to this.

Up to reordering the indices of the $B_j$'s, we can assume that 
$B_1$ is contained in some of the $A_i$'s, and define $\HH_1:=\{A_i\in\HH\st B_1\subseteq A_i\}$. If $\HH_1=\HH$ we are done, otherwise  there is another $B_j$ contained in the elements of $\HH\setminus \HH_1$, and up to reordering we assume that this is $B_2$. We define $\HH_2:=\{A_i\in \HH\setminus \HH_1\st B_2\subseteq A_i\}$. If $\HH\setminus (\HH_1\cup \HH_2)$ is not empty we repeat the process with $B_3$ and obtain $\HH_3$.
We stop after $\ell$ steps, where $\ell\leq r$, when we have $\HH_1\cup\dotsb\cup \HH_{\ell}=\HH$.

By construction we have:
\begin{equation}\label{eq_inclusion_H_B}
B_1\subseteq \scap\HH_1, \ldots, B_{\ell}\subseteq\scap\HH_{\ell},
\end{equation}
so
\[
(\scap \HH_1)\cap\dotsb\cap(\scap\HH_{\ell})=\scap\HH=A\subseteq B_1\cap\dotsb\cap B_{\ell}\subseteq (\scap \HH_1)\cap\dotsb\cap(\scap\HH_{\ell})
\]
and this implies that $A=B_1\cap\dotsb\cap B_{\ell}$, that is to say, $r=\ell$.

Recall that $B_1\cap\dotsb\cap B_{\ell}$ is transversal and so is $\scap \HH$; since each $\HH_j$ is an antichain, $\scap \HH_j$ is transversal too, and it follows that $(\scap\HH_{1})\cap\dotsb\cap(\scap\HH_{\ell})$ is transversal. From this information and from~\eqref{eq_inclusion_H_B} we deduce that $\dim(B_j)=\dim(\scap\HH_j)$ for all $j=1,\ldots,\ell$, therefore $B_j=\scap\HH_j$. 

This implies that $\card{\HH_j}=1$ for any $j$, otherwise if $\card{\HH_j}>1$ that would contradict the fact that the intersection of an antichain does not belong to $\GG$.

Since $\HH_1\cup\dotsb\cup\HH_{\ell}=\HH=\{A_1,\ldots,A_k\}$ we deduce that $k=\ell$ and, up to a relabeling, we can assume that $\HH_i=\{A_i\}$. We conclude that $A_j=B_j$ and this shows that $A_1,\ldots,A_k$ are the $\GG$-factors of $A$.
\end{proof}

\begin{rmk}
Proposition~\ref{prop_nested_characterization} implies that when $\GG$ is well-connected the property of being nested can be expressed in global terms (without charts).
\end{rmk}

\begin{exam}[Example~\ref{ex_arrangement}, continued]\label{ex_arrangement_nested_sets}
The building set $\GG$ of Example~\ref{ex_arrangement} has 48 nested sets, namely
\[
\begin{array}{cccc}
\emptyset, &
\{\kk_1\}, &
\{\kk_2\}, &
\{\kk_3\}, \displaybreak[0] \\
\{L_2\}, &
\{L_3\}, &
\{P_1\}, &
\{P_2\}, \displaybreak[0] \\
\{P_3\}, &
\{P_4\}, &
\{\kk_1,\kk_2\}, &
\{\kk_1,L_2\}, \displaybreak[0] \\
\{\kk_1,L_3\}, &
\{\kk_1,P_1\}, &
\{\kk_1,P_2\}, &
\{\kk_1,P_3\}, \displaybreak[0] \\
\{\kk_1,P_4\}, &
\{\kk_2,\kk_3\}, &
\{\kk_2,P_1\}, &
\{\kk_2,P_2\}, \displaybreak[0] \\
\{\kk_2,P_3\}, &
\{\kk_2,P_4\}, &
\{\kk_3,L_2\}, &
\{\kk_3,L_3\}, \displaybreak[0] \\
\{\kk_3,P_1\}, &
\{\kk_3,P_2\}, &
\{\kk_3,P_3\}, &
\{\kk_3,P_4\}, \displaybreak[0] \\
\{L_2,P_1\}, &
\{L_2,P_3\}, &
\{L_3,P_2\}, &
\{L_3,P_4\}, \displaybreak[0] \\
\{\kk_1,\kk_2,P_1\}, &
\{\kk_1,\kk_2,P_2\}, &
\{\kk_1,\kk_2,P_3\}, &
\{\kk_1,\kk_2,P_4\}, \displaybreak[0] \\
\{\kk_1,L_2,P_1\}, &
\{\kk_1,L_2,P_3\}, &
\{\kk_1,L_3,P_2\}, &
\{\kk_1,L_3,P_4\}, \displaybreak[0] \\
\{\kk_2,\kk_3,P_1\}, &
\{\kk_2,\kk_3,P_2\}, &
\{\kk_2,\kk_3,P_3\}, &
\{\kk_2,\kk_3,P_4\}, \displaybreak[0] \\
\{\kk_3,L_2,P_1\}, &
\{\kk_3,L_2,P_3\}, &
\{\kk_3,L_3,P_2\}, &
\{\kk_3,L_3,P_4\}. \displaybreak[0] \\
\end{array}
\]
\end{exam}

Let \(\GG=\{G_1,G_2,\ldots, G_m\}\) be a well-connected building set and let \(\NS\) be a \(\GG\)-nested set. Given $A\in \NS$, we define $\NS_{A}:=\{B\in \NS\st A\subsetneq B\}$ and for every \(A\in \NS\) we denote by \(M_{\NS}(A)\) the (connected)  intersection $\scap \NS_{A}$. We will omit  the nested set $\NS$ and write just $M(A)$ for brevity when it is clear from the context which is the involved nested set.



\begin{defn}\label{def_admissible}
A function \(f\colon\GG \to \NN\) is (\(\GG\)-){\em admissible}  if it has both  the following properties:
\begin{enumerate}
\item \(\supp  f\) is  \(\GG\)-nested; 
\item for every \(A\in \supp f\) we have \(f(A)<\dim M_{\supp f}(A)- \dim A\).
\end{enumerate}
Notice that the zero function, i.e. the function such that $f(A)=0$ for  
every $A\in\GG$, is admissible since its support is the empty set.
\end{defn}


\begin{exam}[Example~\ref{ex_arrangement}, continued]\label{ex_arrangement_admissible}
For each nested set $\NS$ listed in Example~\ref{ex_arrangement_nested_sets}, we test if it can be the support of an admissible function by computing the maximum value that the candidate function can assume on the elements of $\NS$ (see Definition~\ref{def_admissible}). It turns out that only 7 of the 48 nested sets give rise to admissible functions, namely
\begin{equation}\label{eq_ex_arrangement_supports}
\emptyset,\ \{L_2\},\ \{L_3\},\ \{P_1\},\ \{P_2\},\ \{P_3\},\ \{P_4\}.
\end{equation}
In particular we find 11 admissible functions:
\begin{center}
\begin{tabular}{cc}
\toprule
Support & Values \\
\midrule
$\emptyset$ & $f(G)=0$ for all $G\in\GG$ \\
$\{L_2\}$ & $f(L_2)=1$ \\
$\{L_3\}$ & $f(L_3)=1$ \\
$\{P_1\}$ & $f(P_1)=1$ \\
$\{P_1\}$ & $f(P_1)=2$ \\
$\{P_2\}$ & $f(P_2)=1$ \\
$\{P_2\}$ & $f(P_2)=2$ \\
$\{P_3\}$ & $f(P_3)=1$ \\
$\{P_3\}$ & $f(P_3)=2$ \\
$\{P_4\}$ & $f(P_4)=1$ \\
$\{P_4\}$ & $f(P_4)=2$ \\
\bottomrule
\end{tabular}
\end{center}
\end{exam}


Let us fix some notation. Given a nested set \(\NS\) we know that the intersection $\scap \NS$ is non-empty and it is a layer of type \(\bar{\kk(\Gamma(\NS), \phi_\NS)}\) for some $\Gamma(\NS)$ and $\phi_\NS$. Let $H(\NS)$ be the subtorus associated with $\Gamma(\NS)$ as in Theorem~\ref{thm_layer_chiusura}, namely $H(\NS)=\cap_{\chi\in\Gamma(\NS)}\ker(x_\chi)$.
By Theorem~\ref{thm_layer_chiusura}, given a fan $\Delta$, $\Delta_{H(\NS)}\coloneqq\Delta \cap V_{\Gamma(\NS)}$ is a smooth fan which we denote by \(\subfan{\NS}\). Let \(\XX{\subfan{\NS}}\) be the corresponding toric variety and let \(\pi_\NS\) be the projection \(\pi_\NS \colon H^*(\Xdelta,\ZZ) \to H^*(\XX{\subfan{\NS}},\ZZ)\), which is the restriction map induced by the inclusion, i.e.\ the one of~\eqref{eq_restriction_map}.
Let \(\Theta(\NS)\)  be a minimal set of elements of \(H^*(\Xdelta,\ZZ)\) such that their image via \(\pi_\NS\) is a basis of 
\(H^*(\XX{\subfan{\NS}},\ZZ)\).

Given a (non necessarily admissible) function \(f\colon\GG\to\NN\) we define the monomial in \(H^*(\YY{\GG},\ZZ)\) viewed as $R[T_1,\ldots,T_m]/\II{\GG}$
\[
m_f=\prod_{G_i\in \GG}T_i^{f(G_i)}
\]
where $R=H^*(\Xdelta,\ZZ)$ and $T_i$ is the (class of the) variable associated with $G_i$,
and denote by \(\bb_\GG\) the following set of elements of $H^*(\YY{\GG},\ZZ)$:
\begin{equation}\label{eq:monomialbasis}
\bb_\GG= \{b\,m_f \st f \text{ is admissible}, b \in \Theta({\supp f}) \}.
\end{equation}

\begin{rmk}
For $\NS=\emptyset$, the only admissible function is the zero function. In this case the associated monomial is $1$, and $\Theta(\emptyset)$ is a basis of \(H^*(\Xdelta,\ZZ)\). In fact $\scap \NS=\Xdelta$ by the usual convention, which
is the closure of $\TT$ that, as a layer, has $\Gamma=\{0\}$ and $V_{\Gamma}=V$. So $\Delta(\emptyset)=\Delta$ and $\pi_{\emptyset}$ is the identity function. In particular $\bb_\GG$ contains the set $\{a \cdot 1 \st a\in\Theta(\emptyset) \}$.
\end{rmk}

\begin{thm}\label{teo_main}
Let $\arr$, $\Xdelta$, $\Lambda$ and $\GG$ be as in the beginning of this section. The set \(\bb_\GG\) defined as in~\eqref{eq:monomialbasis} is a \(\ZZ\)-basis of $H^*(\YY{\GG},\ZZ)$.
\end{thm}
\begin{proof}
This proof is divided in two parts: we first show that the elements in 
\(\bb_\GG\) generate $H^*(\YY{\GG},\ZZ)$ as a $\ZZ$-module; then we see that they are independent by counting them.

\noindent\textbf{$\bb_{\GG}$ generates $H^*(\YY{\GG},\ZZ)$ as a $\ZZ$-module}.
First of all notice that the  relations in Theorem~\ref{thm_relcoom} imply that, given an admissible function $f$ and $d \in H^*(\Xdelta,\ZZ)$ such that $\pi_{\supp f}(d)=0$ in 
$H^*(\XX{\subfan{\supp f}},\ZZ)$ then $d\cdot m_{f}=0$ in $H^*(\YY{\GG},\ZZ)$. In fact $d$ is a polynomial in terms of variables $C_{r}$ where $r$ does not belong to $V_{\Gamma(\supp f)}$---this follows from Theorem~\ref{thm_layer_chiusura}.
Therefore we can prove that $\bb_{\GG}$ generates $H^*(\YY{\GG},\ZZ)$ by showing that the set
\[\MM_\GG=\{m_f \st f \text{ is admissible} \}\]
generates $H^*(\YY{\GG},\ZZ)$ as a $H^*(\Xdelta,\ZZ)$-module.
To show this, let us consider a function $g\colon \GG\rightarrow \NN$ which is not admissible: we will prove that $m_{g}$ can be obtained as a 
$H^*(\Xdelta,\ZZ)$-linear combination of monomials in $\MM_{\GG}$.

If $\supp g$ is non-nested then the intersection of the elements in $\supp g$ is empty by Theorem~2.2 of~\cite{deconcgaiffi1}, so $m_g=0$ in $H^*(\YY{\GG},\ZZ)$. From now on we can assume that $\supp g$ is nested, and following the notation introduced before Definition~\ref{def_admissible} we will write $M(B)=M_{\supp g}(B)$.

Let us consider $\GG\cup\{\emptyset\}$ as a set partially ordered by inclusion. 
On this poset we define a ``pseudo-rank'' function $r$ as follows: $r(\emptyset)=0$ and for $G\in \GG$ let $r(G)$ be the maximal length of a chain between $\emptyset$ and $G$ in the Hasse diagram of the poset.

We say that $B\in \supp g$ is a \emph{bad component} for $g$ if $g(B)\geq \dim M(B)-\dim B$. To every bad component $B$ we assign the pair 
\[
\big( r(B),g(B)-(\dim M(B)-\dim B ) \big)\in \NN\times\NN
\] 
and we put on $\NN\times\NN$ the lexicographic order.\footnote{i.e.\ $(a,b)<(c,d)$ if $a<c$ or if $a=c$ and $b<d$.}
Since we assumed that $\supp g$ is nested but that $g$ is not admissible, the set of bad components for $g$ is not empty, so we can define the \emph{evaluation} of $g$ as the maximal pair associated with the bad components of $g$. 
We will proceed by induction on the evaluation.

\noindent
\underline{Base step.} Let us consider a non-admissible function $g$ whose evaluation is $(1,a)$, for some $a\in\NN$. This means that the maximal bad components of $g$ are minimal elements in $\GG$ w.r.t.\  inclusion. Let $B$ be such a bad component. We can partition $\supp g$ as $\NS\cup \NS'\cup \{B\}$ where $\NS:=(\supp g)_B$ and $\NS':=\supp g\setminus (\NS\cup\{B\})$. The monomial $m_g$ associated with $g$ is of the following type:
\[
m_g=\prod_{K\in\NS'}T_K^{g(K)}\prod_{K\in\NS}T_K^{g(K)}\cdot T_B^{g(B)}.
\]
From Theorem~\ref{thm_relcoom} we know that in the ideal $\II{\GG}$ there is the element
\[
F(B,\NS)=P^{M}_B(-T_B)\prod_{K\in \NS}T_K
\]
where $M=M_{\NS}(B)=M_{\supp g}(B)=M(B)$.

Now we notice that $P^{M}_B(-T_B)$ is a polynomial in $H^*(\Xdelta,\ZZ)[T_B]$ of the following form:
\[
\pm (T_B)^{\dim M-\dim B}+ \textrm{lower order terms in }T_B.
\]

By writing $m_g$ as 
\[
m_g=\Big(\prod_{K\in\NS'}T_K^{g(K)}\prod_{K\in\NS}T_K^{g(K)-1}\cdot T_B^{g(B)-(\dim M-\dim B)}\Big)\Big(\prod_{K\in\NS}T_K\cdot T_B^{\dim M-\dim B} \Big),
\]
we see that the leading term of $F(B,\NS)$, namely $T_B^{\dim M-\dim B}\prod_{K\in \NS}T_K$, divides $m_g$, so when we reduce $m_g$ modulo $\II{\GG}$ we obtain a polynomial of the following form:
\begin{align*}
&\Big(\prod_{K\in\NS'}T_K^{g(K)}\prod_{K\in\NS}T_K^{g(K)-1}\cdot T_B^{g(B)-(\dim M-\dim B)}\Big)\Big(\prod_{K\in\NS}T_K\cdot\sum_{k=0}^{\dim M-\dim B-1} c_k\cdot T_B^{k}\Big)\\
=&\sum_{k=0}^{\dim M-\dim B-1} c_k \left(\prod_{K\in\supp g\setminus\{B\}}T_K^{g(K)}\right)\cdot T_B^{g(B)-(\dim M-\dim B)+k}=\sum_{k=0}^{\dim M-\dim B-1} c_k m_{g_k}
\end{align*}
for some $c_k\in H^*(\Xdelta,\ZZ)$ and some suitable $g_k\colon \GG\to \NN$. Notice that $\supp g_k=\supp g$ (eventually without $B$) and that the $g_k$'s coincide with $g$ in $\supp g\setminus\{B\}$, whereas
\[
g_k(B)=g(B)-(\dim M-\dim B)+k.
\]

Therefore every monomial $m_{g_k}$ appearing in this formula either is in $\bb_{\GG}$ or its evaluation is $(1,b)$ with $b=a+k-(\dim M-\dim B)<a$. 
We can apply the same argument to the latter monomials until we get a linear combination of monomials in $\bb_{\GG}$.

\noindent
\underline{Inductive step.} Let us suppose that our claim is true for non-admissible functions (with nested support) whose evaluation is $(k,c)$ with $k>1$ and $c\in\NN$. Let us consider a non-admissible function $g$ with evaluation $(k+1,a)$. This means that there is at least one bad component $B$ whose  associated pair is $(k+1,a)$.

As before we consider the element of $\II{\GG}$:
\[
F(B,\NS)=P^{M}_B(\sum_{D\in \GG, D\subseteq B}-T_D)\prod_{K\in \NS}T_K
\]
where $\NS=(\supp g)_B$ and $M=M_{\NS}(B)=M_{\supp g}(B)=M(B)$.
The polynomial $P^{M}_B(\sum -T_D)$ is of type
\[
\pm (T_B)^{\dim M-\dim B}+ q,
\]
where $q$ is a polynomial in $H^*(\Xdelta,\ZZ)[T_{D}\st  D\in \GG, D\subseteq B]$ with degree in $T_B$ strictly less than $\dim M-\dim B$.

As in the base step, we notice that the leading term of $F(B,\NS)$ divides $m_g$ and this allows us to write $m_g$ modulo $\II{\GG}$ as a $H^*(\Xdelta,\ZZ)$-linear combination of monomials. If these are not $0$ modulo $\II{\GG}$ either they are in $\bb_{\GG}$ or their evaluation is strictly less than $(k+1,a)$.
If some of them have evaluation $(k+1,b)$ with $b<a$
 we can again use the relations in $\II{\GG}$ and in a finite number of steps we get a $H^*(\Xdelta,\ZZ)$-linear combination of monomials that are in $\bb_{\GG}$ or have evaluation $(s,t)$ with $s\leq k$. To these latter monomials we can apply  the inductive hypothesis.

\noindent
\textbf{$\bb_{\GG}$ is a set of independent elements}.
To show that the elements in $\bb_{\GG}$ are linearly independent over $\ZZ$ it suffices to show that $\card{\bb_{\GG}}$ is equal to the rank of $H^*(\YY{\GG},\ZZ)$ which is a free $\ZZ$-module (see Theorem~\ref{thm_relcoom}).
We proceed by induction on $m$, the cardinality of ${\GG}$.

\noindent
\underline{Base step.} If $m=1$ then $\GG=\{G_1\}$ and $\YY{\GG}$ is the blowup of $\Xdelta$ along $G_1$. As it is well known (see for example~\cite[Chapter~4, Section~6]{griffithsharris}) we have the following isomorphism of graded $\ZZ$-modules:
\[
H^*(\YY{\GG},\ZZ)\cong H^*(\Xdelta,\ZZ)\oplus\bigoplus_{J=1}^{\codim G_1-1}H^*(G_1,\ZZ)\zeta^J
\]
where $\zeta^J$ is a symbol that shifts the degrees of $+2J$.
We now split $\bb_{\GG}$ into two disjoint subsets:
\[
\bb_{\GG}^1:=\{a\cdot 1\st a \in \Theta({\emptyset})\}
\]
\[
\bb_{\GG}^2:=\{b\,T_{1}^r\st r=1,\ldots,\codim G_1-1\text{ and }b \in \Theta(\{G_1\})\}
\]
We observe that by Theorem~\ref{thm_layer_chiusura}, point~\ref{thm_layer_chiusura_iii},
$G_1$ is isomorphic to $\XX{\subfan{\{G_1\}}}$.
Therefore there is a grade-preserving bijection between $\bb_{\GG}^2$ and a basis of $\bigoplus H^*(G_1,\ZZ)\zeta^J$ which, together with the known bijection between $\bb_{\GG}^1$ and a basis of $H^*(\Xdelta,\ZZ)$, proves our claim in this case.

\noindent
\underline{Inductive step.}
We assume that the claim holds for every toric model associated with a building set of cardinality less than or equal to $m-1$. Consider $\GG=\{G_1,\ldots,G_m\}$ and assume that the labelling is a refinement of the ordering in $\GG$ by inclusion. 
We put $\GG_{i}:=\{G_1,\ldots,G_i\}$ for every $i\leq m$ and $Z:=G_{m}$; we denote then $\dt{Z}$ the proper 
transform of $Z$ in the variety $\YY{\GG_{m-1}}$. Then $\YY{\GG}$ is obtained as the blow up of $\YY{\GG_{m-1}}$ along $\dt{Z}$.

We can use again the result from~\cite{griffithsharris} and obtain the following graded isomorphism of $\ZZ$-modules:
\[
H^*(\YY{\GG},\ZZ)\cong H^*(\YY{\GG_{m-1}},\ZZ)\oplus \bigoplus_{J=1}^{\codim G_m-1}H^*(\dt{Z},\ZZ)\zeta^{J}.
\]
Following the idea from the base step, we write $\bb_{\GG}$ as the union of the disjoint sets:
\begin{align*}
\bb_{\GG}^1&{}=\{b\,m_f \st f\text{ admissible, }f(G_m)=0,\ b \in \Theta(\supp f)\},\\
\bb_{\GG}^2&{}=\{b\,m_f\st f\text{ admissible, }f(G_m)\neq0,\  b \in \Theta(\supp f)\}.
\end{align*}

There is a bijective correspondence, provided by the restriction, between the set $\{f\colon\GG\to\NN\st f \text{ admissible, }f(G_m)=0\}$
and the set $\{f\colon\GG_{m-1}\to\NN\st f\text{ admissible}\}$.
By the inductive hypothesis, $\bb_{\GG}^1$ is in bijection with the basis $\bb_{\GG_{m-1}}$ of $H^*(\YY{\GG_{m-1}},\ZZ)$ and this correspondence is grade-preserving.

Notice that because $G_m$ is maximal in $\GG$, then $G_m$ is maximal in $\supp f$ for every admissible function $f$ such that $f(G_m)\neq0$. So the possible values for $f(G_m)$ are $1,\ldots, \codim G_m-1$.

Now we observe that, given an element $b\,m_f\in\bb_{\GG}^2$ (so that $f(G_m)\neq0$), we also have in $\bb_{\GG}^2$ the monomials $b\,m_g$ for all the admissible functions $g$ that coincide with $f$ on $\GG\setminus \{G_m\}$ and such that $g(G_m)\in \{1,\ldots,\codim G_m-1\}\setminus \{f(G_m)\}$.
As a consequence the sets $\{b\,m_f\in\bb_{\GG}^2\st f(G_m)=i\}$, for $i=1,\ldots, \codim G_{m}-1$ have all the same cardinality and form a partition of $\bb_{\GG}^2$, and it suffices to prove that there is a grade-preserving (up to a shift by 2 in cohomology) bijection between $\{b\,m_f\in\bb_{\GG}^2\st f(G_m)=1\}$ and a basis of $H^*(\dt{Z},\ZZ)$. This extends to a grade-preserving bijection between $\bb_{\GG}^2$ and a basis of $\bigoplus H^*(\dt{Z},\ZZ)\zeta^J$.
 
Let us now recall the following result and notation from~\cite[Section~4]{deconcgaiffi2}.
We consider the family $\HH$  of subvarieties in $Z$ that are the connected components of the intersections $G_i\cap Z$ for every $i=1,\ldots,m-1$. Since $\GG$ is well-connected, if $G_i\cap Z$ is not empty and not connected then its connected components belong to $\GG$. This implies that $u:=\card{\HH}\leq m-1$. Now for each $H\in\HH$, we denote by $s(H)$ the minimum index $i$ such that $H$ is a connected component of $G_{i}\cap Z$ (in particular $H=G_{s(H)}\cap Z$). We sort the set $\{s(H)\st H\in\HH \}$ in ascending order as $\{s_{1},\ldots,s_{u}\}$ and let $\HH=\{H_1,\ldots,H_{u}\}$.
\begin{rmk}
Notice that two possibility occurs for $H\in \HH$: either $H=G_{s(H)}$ in case $G_{s(H)}\subset Z$, or $H=G_{s(H)}\cap Z$ and the intersection is transversal.
\end{rmk}

In~\cite[Proposition~4.4]{deconcgaiffi2} it is proven that $\HH$ is building and well-connected and from Proposition~4.6 of the same paper it follows that $\dt{Z}$ is isomorphic to the model $\YY[Z]{\HH}$ obtained by blowing up $\HH$ in $Z$.
From Theorem~\ref{thm_layer_chiusura}, point~\ref{thm_layer_chiusura_ii}, we know that $Z=G_m$ is a toric variety with fan $\Delta'=\subfan{\{G_m\}}$, so $Z=\XX{\Delta'}$.

In analogy with the previous notation, for an $\HH$-nested $\NS$ let $\pi'_{\NS}$ be the projection
\(\pi'_\NS \colon H^*(\XX{\Delta'},\ZZ) \to H^*(\XX{\Delta'(\NS)},\ZZ)\)
and we take $\Theta'(\NS)$ as a minimal set of elements of \(H^*(\XX{\Delta'},\ZZ)\) such that their image via \(\pi'_\NS\) is a basis of 
\(H^*(\XX{\Delta'(\NS)},\ZZ)\).

Since $\card{\HH}< m$ we can apply our inductive hypothesis to $\YY[Z]{\HH}$ and we get the following $\ZZ$-basis $\bb_{\HH}$ of $H^*(\YY[Z]{\HH},\ZZ)$:
\[
\bb_{\HH}=\{b'\,m_g\st\ g \ \HH \text{-admissible},\ b'\in \Theta'(\supp g)\}.
\]
We are now ready to describe a bijective, grade-preserving (up to a shift by 2 in cohomology) correspondence between $\bb_{\HH}$ and
\[
L:=\{b\,m_f\st f\ \GG\text{-admissible},\ f(G_m)=1,\ b \in \Theta(\supp f)\}\subset \bb_{\GG}^2.
\]

Let $f\colon \GG\to \NN$ be a $\GG$-admissible function with $f(G_m)=1$. We associate to $f$ the function $\bar{f}\colon\HH\to\NN$ such that $\bar{f}(H)=f(G_{s(H)})$ for all $H\in \HH$. 
\begin{lemma}
The function $\bar{f}$ is $\HH$-admissible.
\end{lemma}
\begin{proof}
If $\supp \bar{f}$ is empty, $\bar{f}$ is admissible by definition,
so from now we suppose that $\supp \bar{f}\neq \emptyset$.

We show that $\supp \bar{f}$ is $\HH$-nested by using the characterization from Proposition~\ref{prop_nested_characterization}:
we take an antichain $\mathcal{K}=\{K_1,\ldots, K_{t}\}$
($t\geq 2$) in $\supp \bar{f}$ and prove that the intersection $\scap \mathcal{K} $
is connected, transversal and does not belong to $\HH$.

\noindent
\underline{$\scap \mathcal{K}$ is connected.} It is equal to $G_{s(K_1)}\cap\dotsb\cap G_{s(K_t)}\cap G_m$ and they all belong to $\supp f$ which is $\GG$-nested. In particular, their intersection is the intersection of the minimal elements among them, that is to say, the intersection of an antichain of a nested set, which is connected by Proposition~\ref{prop_nested_characterization}.

\noindent
\underline{$\scap \mathcal{K}$ is transversal.} We split this part of the proof in two cases:
\begin{enumerate}
\item if there is a $j\in\{1,\ldots,t\}$ such that $K_j=G_{s(K_j)}$, 
then 
\[
K_1\cap\dotsb\cap K_j\cap\dotsb\cap K_t=G_m\cap G_{s(K_1)}\cap\dotsb\cap G_{s(K_t)}=G_{s(K_1)}\cap\dotsb\cap G_{s(K_t)}
\]
which is transversal because $\{G_{s(K_1)},\ldots,G_{s(K_t)}\}$ is a nested set;
\item if $K_j=G_m\cap G_{s(K_j)}$ transversally for every $j\in \{1,\ldots,t\}$, then the set $\{G_{s(K_1)},\ldots, G_{s(K_t)},G_m\}$ is a set of elements which are pairwise non-comparable (the $G_{s(K_i)}$'s are not pairwise comparable because the $K_i$'s are not), so their intersection 
\[
G_m\cap G_{s(K_1)}\cap\dotsb\cap G_{s(K_t)}=\scap \mathcal{K}
\]
 is transversal.

\end{enumerate}

\noindent
\underline{$\scap \mathcal{K}$ does not belong to $\HH$.}
Suppose that $\scap \mathcal{K}=H\in\HH$. It is not possible that $H=G_{s(H)}$, because otherwise $\scap \mathcal{K}$ would belong to $\GG$ in contradiction with Proposition~\ref{prop_nested_characterization}. On the other hand, if $H=G_{s(H)}\cap G_m$ transversally, as a consequence of~\cite[Proposition~3.3]{deconcgaiffi2} we have that $G_{s(H)}$ and $G_m$ are the minimal elements among the ones in $\GG$ containing $H$. We study the following two subcases:
\begin{itemize}
\item if there is a $j\in\{1,\ldots,t\}$ such that $K_j=G_{s(K_j)}$, then we would have $H\subset G_{s(K_j)}\subset G_m$ in contradiction with the minimality of $G_m$;
\item if $K_j=G_m\cap G_{s(K_j)}$ transversally for every $j\in \{1,\ldots,t\}$, since $k\geq 2$ we would have two different $\GG$-decompositions of $H$, namely
\[
H=G_m\cap G_{s(K_j)}=G_m\cap G_{s(K_1)}\cap\dotsb\cap G_{s(K_t)}.
\]
\end{itemize} 
This concludes the proof that $\supp \bar{f}$ is nested.

Now let us consider $H\in\supp \bar{f}$. We study the following two cases:

\noindent
\underline{Case $H=G_{s(H)}$.}
In this case $\bar{f}(H)=f(G_{s(H)})$ and, since $f$ is $\GG$-admissible, we have 
\[
1\leq f(G_{s(H)})<\dim M_{\supp f}(G_{s(H)})-\dim G_{s(H)}.
\] 
Now, because $G_m\in \{B\in\supp f\st  G_{s(H)}\subsetneq  B\}$ we notice that 
\[
\scap (\supp f)_{G_{s(H)}}=\scap (\supp \bar{f})_{H} 
\]
where $\supp \bar{f}$ is viewed as a $\HH$-nested set and
with the usual convention that if $ \{H_i\in\supp \bar{f}\st  H\subsetneq  H_i\}$ is empty the intersection is the ambient space $G_m$. In particular
\[
\dim \big(M_{\supp f}(G_{s(H)})\big)=\dim \big( M_{\supp \bar{f}}(H)\big),
\]
therefore $\bar{f}(H)$ ranges in the expected interval.

\noindent
\underline{Case $H=G_{s(H)}\cap G_m$ transversally.}
In this case $\bar{f}(H)=f(G_{s(H)})$ where again 
\[
1\leq f(G_{s(H)})<\dim M_{\supp f}(G_{s(H)})-\dim G_{s(H)}.
\] 
Now since $\supp f$ is $\GG$-nested and $G_m$ does not belong to $\{B\in\supp f\st  G_{s(H)}\subsetneq  B\}$ we have that
\begin{equation}\label{Eq:uguaglianza_M}
\dim \left(M_{\supp f}(G_{s(H)}) \cap G_m\right)=\dim \left( M_{\supp f}(G_{s(H)})\right)  -\codim G_m.
\end{equation}
But 
\[
\Big(\scap(\supp f)_{G_{s(H)}} \Big) \cap G_m=\scap(\supp \bar{f})_{H} 
\]
so we can rewrite \eqref{Eq:uguaglianza_M} as
\[
\dim M_{\supp \bar{f}}(H)=\dim \left( M_{\supp f}(G_{s(H)})\right)  -\codim G_m
\]
and, observing that $\dim H=\dim G_{s(H)}-\codim G_m$, we conclude that 
\begin{align*}
\dim M_{\supp f}(G_{s(H)})-\dim G_{s(H)}
&{}=\dim M_{\supp \bar{f}}(H)+\codim G_m-\dim G_{s(H)}\\
&{}=\dim M_{\supp \bar{f}}(H)-\dim H
\end{align*}
therefore also in this case $\bar{f}(H)$ ranges in the expected interval.
\end{proof}

\begin{lemma}\label{lemma_supporto_f}
If $f$ is $\GG$-admissible and $f(G_m)=1$, then $\supp f\setminus\{G_m\}\subseteq \{G_{s(H)}\st H\in \HH\}$.
\end{lemma}
\begin{proof}
Let us suppose $B=G_{k}\in \supp f\setminus\{G_m\}$ with $k\neq s(H)$ for every $H\in \HH$. We notice that $B\nsubseteq G_m$ otherwise $B=B\cap G_m\in \HH$ so $B=G_{s(B)}$. Moreover, since $\supp f$ is nested and contains both $B$ and $G_m$ it follows that $B\cap G_m\neq \emptyset$ and connected by Proposition~\ref{prop_nested_characterization}. 
Now we observe that from the connectedness of $B\cap G_m$ and the definition of $\HH$ we have that $B\cap G_m=H_j$ for some $j$ and that $B$ and $G_m$ are its $\GG$-factors. But $H_j=G_{s_j}\cap G_m$ is a different $\GG$-factorization of $H_j$, obtaining a contradiction.  
\end{proof}

Therefore, given $f$ $\GG$-admissible with $f(G_m)=1$ we can associate two monomials: $m_f$ and $m_{\bar{f}}$. Now in $\bb_{\GG}$ we find elements of the form $b\,m_{f}$ with $b$ belonging to $\Theta(\supp f)$; on the other hand in $\bb_{\HH}$ we find elements of the form $b'\,m_{\bar{f}}$ with $b'$ belonging to $\Theta'(\supp \bar{f})$.
But $\Delta'(\supp \bar{f})=\subfan{\supp f}$, therefore we can choose $b$ and $b'$ above so that they range over the same set.

We have thus constructed a map $\Phi\colon L \to\bb_{\HH}$ such that
$\Phi(b\,m_f)=b\,m_{\bar{f}}$. If we show that $\Phi$ is a bijection, this concludes the proof of the theorem. Actually it is sufficient to prove that $\Phi$ is injective: in fact the injectivity implies
\[
\card{L}\leq \card{\bb_{\HH}}= \rk(H^*(\dt{Z},\ZZ))
\]
 where the last equality, as we have seen, derives from the inductive hypothesis, since $\bb_{\HH}$ is a basis of $H^*(\dt{Z},\ZZ)$. This in turn implies that 
 \begin{align*}
 \card{\bb_{\GG}}&{}=\card{\bb_{\GG}^1}+\card{\bb_{\GG}^2}=\card{\bb_{\GG}^1}+(\codim G_m-1)\card{L}=\\
 &{}=\rk H^*(\YY{\GG_{m-1}},\ZZ)+(\codim G_m-1)\card{L}\\
 &{}\leq \rk H^*(\YY{\GG_{m-1}},\ZZ)+(\codim G_m-1)\rk(H^*(\dt{Z},\ZZ))
 =\rk H^*(\YY{\GG},\ZZ).
 \end{align*}
 On the other hand we already know, from the first part of this proof, that $\bb_{\GG}$ generates $H^*(\YY{\GG},\ZZ)$. It follows that 
\[
\card{\bb_{\GG}}=\rk H^*(\YY{\GG},\ZZ)
\] 
which is the claim of the theorem (and of course this also implies that $\Phi$ is actually a bijection). 

To show the injectivity of $\Phi$ let $f_1$, $f_2$ be two distinct $\GG$-admissible functions with $f_1(G_m)=f_2(G_m)=1$: we prove that $\bar{f_1}\neq \bar{f_2}$.

Let us first suppose that $\supp f_1\neq \supp f_2 $; up to switching $f_1$ and $f_2$, we can assume that there exists $B\in \supp f_1\setminus \supp f_2$. By Lemma~\ref{lemma_supporto_f} we know that $B=G_{s(H)}$ for a certain $H\in \HH$. We deduce that $H\in \supp \bar{f_1}\setminus \supp \bar{f_2}$ and conclude that $\bar{f_1}\neq \bar{f_2}$.

If instead $\supp f_1=\supp f_2$, $f_1\neq f_2$ implies that there is a certain $B$ in their support such that $f_1(B)\neq f_2(B)$.
Again by Lemma~\ref{lemma_supporto_f} we know that $B=G_{s(H)}$, $H\in \HH$, therefore $\bar{f_1}(H)=f_1(G_{s(H)})\neq f_2(G_{s(H)})=\bar{f_2}(H)$. This proves that $\Phi$ is injective and concludes the proof of the theorem.
\end{proof}

\begin{exam}[Example~\ref{ex_arrangement}, continued]\label{ex_arrangement_basis}
We can compute a $\ZZ$-basis for the ring $H^*(Y_\arr,\ZZ)$ using the admissible functions found in Example~\ref{ex_arrangement_admissible}. The result is detailed in Tables~\ref{tab_example_basis_1} and~\ref{tab_example_basis_2}; the tables have one line for each possible support $\NS$ of admissible functions, as listed in~\eqref{eq_ex_arrangement_supports}.
\begin{table}[!htb]
\caption{Data used to build the basis of $H^*(Y_\arr,\ZZ)$.}
\label{tab_example_basis_1}
\centering
\begin{tabular}{ccc}
\toprule
\multirow{2}{*}{$\NS$} & Basis for & Monomials $m_f$ \\
& $H^*(\XX{\subfan{\NS}},\ZZ)$ & with $f$ s.t.\ $\supp f=\NS$  \\
\midrule
$\emptyset$ & Basis of $H^*(\Xdelta,\ZZ)$ & $\{1\}$ \\
$\{L_2\}$ & $\{C_7,1\}$ & $\{T_4\}$ \\
$\{L_3\}$ & $\{C_7,1\}$ & $\{T_5\}$ \\
$\{P_1\}$ & $\{1\}$ & $\{T_6,T_6^2\}$ \\
$\{P_2\}$ & $\{1\}$ & $\{T_8,T_8^2\}$ \\
$\{P_3\}$ & $\{1\}$ & $\{T_7,T_7^2\}$ \\
$\{P_4\}$ & $\{1\}$ & $\{T_9,T_9^2\}$ \\
\bottomrule
\end{tabular}
\end{table}

\begin{table}[!htb]
\caption{Contribution to the basis of $H^*(Y_\arr,\ZZ)$ and to the Betti numbers of $Y_\arr$.}
\label{tab_example_basis_2}
\centering
\begin{tabular}{cccccc}
\toprule
\multirow{2}{*}{$\NS$} & Contribution to & \multicolumn{4}{c}{Contribution to $\rk(H^i(Y_\arr,\ZZ))$} \\
& the basis $\bb_\GG$ & $i=0$ & $i=2$ & $i=4$ & $i=6$ \\
\midrule
$\emptyset$ &  Basis of $H^*(\Xdelta,\ZZ)$ & 1 & 69 & 69 & 1 \\
$\{L_2\}$ & $\{C_7T_4,T_4\}$ & 0 & 1 & 1 & 0 \\
$\{L_3\}$ & $\{C_7T_5,T_5\}$ & 0 & 1 & 1 & 0 \\
$\{P_1\}$ & $\{T_6,T_6^2\}$ & 0 & 1 & 1 & 0 \\
$\{P_2\}$ & $\{T_8,T_8^2\}$ & 0 & 1 & 1 & 0 \\
$\{P_3\}$ & $\{T_7,T_7^2\}$ & 0 & 1 & 1 & 0 \\
$\{P_4\}$ & $\{T_9,T_9^2\}$ & 0 & 1 & 1 & 0 \\
\midrule
& $\bb_\GG$ & 1 & 75 & 75 & 1 \\
\bottomrule
\end{tabular}
\end{table}
\end{exam}

\begin{exam}\label{ex_arr2}
As we have seen in Example~\ref{ex_arrangement_admissible}, not all the nested sets are supports of admissible functions. In particular, for small toric arrangements in low dimensions admissible functions often are supported only on singletons. In this example we study a case where there are supports of admissible functions with cardinality 2.

Let $\arr=\{\kk_1,\kk_2,\kk_3\}$ be the arrangement in $(\CC^*)^4$, with coordinates $(x,y,z,t)$, whose layers are defined by the equations
\[
\kk_1\colon z=t=1,\qquad
\kk_2\colon y=1,\qquad
\kk_3\colon x=1.
\]
The poset of layers $\play(\arr)$ is represented in Figure~\ref{fig_arr2_poset}. In this example $\Xdelta=(\PP^1)^4$ is a good toric variety for the arrangement, so we can use its associated fan $\Delta$ (recall that its 16 maximal cones are $C(\sigma_1e_1,\sigma_2e_2,\sigma_3e_3,\sigma_4e_4)$ where $e_1,\dotsc,e_4$ are the vectors of the canonical basis of $\CC^4$ and $(\sigma_1,\sigma_2,\sigma_3,\sigma_4)\in\{\pm1\}^4$); moreover we choose $\GG=\play_0(\arr)$ and build the model $Y_\arr=\YY{\GG}$.
\begin{figure}[!htb]
\centering
\begin{tikzpicture}[baseline=(ref.base)]
\def\hunit{1}
\def\vunit{1}
\tikzset{poset element/.style={draw,circle,fill=white,inner sep=1pt,minimum size=5mm,font=\scriptsize}}
\tikzset{building element/.style={poset element,double}}
\node (ref) at (0*\hunit,2*\vunit) {\phantom{$\kk_1$}};
\node[poset element] (v0) at (0*\hunit,0*\vunit) {$\TT$};
\node[poset element] (v1) at (-1.5*\hunit,2*\vunit) {$\kk_1$};
\node[poset element] (v2) at (0*\hunit,1*\vunit) {$\kk_2$};
\node[poset element] (v3) at (1*\hunit,1*\vunit) {$\kk_3$};
\node[poset element] (v4) at (-1*\hunit,3*\vunit) {$M_1$};
\node[poset element] (v5) at (0*\hunit,3*\vunit) {$M_2$};
\node[poset element] (v6) at (1.5*\hunit,2*\vunit) {$L$};
\node[poset element] (v7) at (0*\hunit,4*\vunit) {$P$};
\draw[->] (v0) -- (v1);
\draw[->] (v0) -- (v2);
\draw[->] (v0) -- (v3);
\draw[->] (v1) -- (v4);
\draw[->] (v1) -- (v5);
\draw[->] (v2) -- (v4);
\draw[->] (v2) -- (v6);
\draw[->] (v4) -- (v7);
\draw[->] (v3) -- (v5);
\draw[->] (v3) -- (v6);
\draw[->] (v5) -- (v7);
\draw[->] (v6) -- (v7);
\end{tikzpicture}\hspace{1cm}
\begin{tabular}{c@{${}\rightsquigarrow{}$}c}
$\kk_1$ & $T_1$ \\
$\kk_2$ & $T_2$ \\
$\kk_3$ & $T_3$ \\
$M_1$ & $T_4$ \\
$M_2$ & $T_5$ \\
$L$ & $T_6$ \\
$P$ & $T_7$
\end{tabular}
\caption{Poset of layers $\play(\arr)$ for the arrangement of Example~\ref{ex_arr2}. Elements at the same height have the same codimension. On the right: map from elements of $\GG=\play_0(\arr)$ to the corresponding variables $T_i$ in $H^*(\Xdelta,\ZZ)[T_1,\dotsc,T_7]$.}
\label{fig_arr2_poset}
\end{figure}
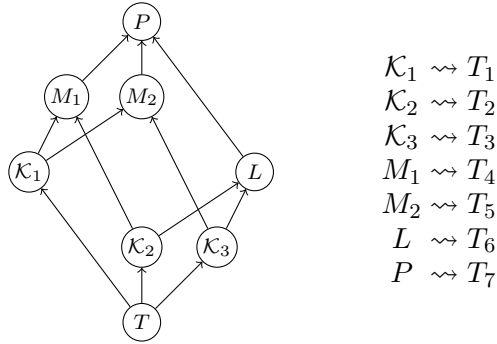

The possible non-empty supports of admissible functions are
\begin{itemize}
\item $\{L\}$: it supports one admissible function $f$ such that $f(L)=1$;
\item $\{M_1\}$, $\{M_2\}$: each supports two admissible functions, namely $f(M_i)=1$ and $f(M_i)=2$;
\item $\{P\}$: it supports three admissible functions, where $f(P)$ is either $1$, $2$ or $3$;
\item $\{L,P\}$: it supports one admissible function such that $f(L)=1$ and $f(P)=1$.
\end{itemize}
Table~\ref{tab_arr2_basis} details the contribution to the Betti numbers for each admissible function.
\begin{table}[!htb]
\caption{Contribution to the Betti numbers of $H^*(Y_\arr,\ZZ)$ for each admissible function, grouped by support.}
\label{tab_arr2_basis}
\footnotesize
\centering
\begin{adjustbox}{max width=1.2\textwidth,center}
\begin{tabular}{cccccccc}
\toprule
\multirow{2}{*}{$\NS$} & Betti numbers & Monomials $m_f$ with $f$ & \multicolumn{5}{c}{Contribution to $\rk(H^i(Y_\arr,\ZZ))$} \\
& for $H^*(\XX{\subfan{\NS}},\ZZ)$ & s.t.\ $\supp f=\NS$ & $i=0$ & $i=2$ & $i=4$ & $i=6$ & $i=8$ \\
\midrule
$\emptyset$ & 1, 4, 6, 4, 1 & $\{1\}$ & 1 & 4 & 6 & 4 & 1 \\
$\{L\}$ & 1, 2, 1 & $\{T_6\}$ & 0 & 1 & 2 & 1 & 0 \\
$\{M_1\}$ & 1, 1 & $\{T_4,T_4^2\}$ & 0 & 1 & 2 & 1 & 0 \\
$\{M_2\}$ & 1, 1 & $\{T_5,T_5^2\}$ & 0 & 1 & 2 & 1 & 0 \\
$\{P\}$ & 1 & $\{T_7,T_7^2,T_7^3\}$ & 0 & 1 & 1 & 1 & 0 \\
$\{L,P\}$ & 1 & $\{T_6 T_7\}$ & 0 & 0 & 1 & 0 & 0 \\
\bottomrule
\end{tabular}
\end{adjustbox}
\end{table}
\end{exam}

\section{The case of root systems of type \texorpdfstring{$A$}{A}}\label{Sec:esempioAn}
In this section we will apply our main theorem to the case of the toric arrangement associated with a root system of type $A_{n-1}$.

\subsection{The minimal toric model and its cohomology basis}\label{subsec:AnBasis}
The toric analogue of the hyperplane arrangement of type $A_{n-1}$ is $\arrA{n-1}=\{\kk_{ij}\st{1\leq i<j\leq n}\}$ in $\TT=(\mathbb{C}^*)^n/H_n\simeq(\CC^*)^{n-1}$, where $H_n$ is the $\CC^*$-span of $(1,\dotsc,1)$ and
\[
\kk_{ij}:=\{[t_1,\ldots,t_n]\in \TT\st t_it_j^{-1}=1\}.
\]
Its poset of intersections $\play(\arrA{n-1})$ is isomorphic to the poset of partitions of the set $\{1,\ldots,n\}$ ordered by refiniment. More precisely, the partition $\{I_1,\ldots,I_k\}$ with $I_1\sqcup\dotsb \sqcup I_k=\{1,\ldots,n\}$ corresponds to the layer
\[
\kk_{I_1,\ldots,I_k}:=\{[t_1,\ldots,t_n]\in \TT\st t_i=t_j\text{ if }\exists\;l\text{ such that }i,j\in I_l\},
\]
where for the sake of convenience we will sometimes omit to write the blocks $I_j$ with cardinality one.

Let $\FFA{n-1}\subset \play(\arrA{n-1})$ be the set whose elements are the $\kk_{I}$ for every $I\subset \{1,\ldots,n\}$ with $\card{I}\geq 2$. It is a building set, in fact it is the minimal one that contains the layers $\kk_{ij}$; it is the analogue of the ``building set of irreducible elements'' for the linear case (see~\cite{wonderful1,yuzvinsky, gaiffiselecta}).

For the arrangement $\arrA{n-1}$ there is a natural choice of a fan that produces a good toric variety, as noted in~\cite{deconcgaiffi1}: we take in $V=X_*(\TT)\otimes_{\ZZ}\RR$ the fan $\DeltaA{n-1}$ induced by the Weyl chambers of the root system. By construction every layer of $\play(\arrA{n-1})$ has an equal sign basis with respect to $\DeltaA{n-1}$, so the toric variety $\XdeltaA{n-1}$ associated with $\DeltaA{n-1}$ is a good toric variety for the arrangement.

In~\cite{procesi90} Procesi studies this toric variety (and also the more general toric varieties $X_{W}$ associated with the fan induced by the Weyl chambers of a Weyl group $W$; see also~\cite{dolgachevlunts}). As it is well-known, the even Betti numbers of the toric variety $\XdeltaA{n-1}$ are the Eulerian numbers $A(n,k)$ (see, for example,~\cite{stembridge92,stanley,stanley-enum1}). We recall briefly the main definitions and results about the numbers $A(n,k)$ and the cohomology of $\XdeltaA{n-1}$.
\begin{defn}
The \emph{Eulerian number} $A(n,k+1)$ is the number of permutations in $S_n$ with $k$ descents\footnote{If $\sigma$ is a permutation in $S_n$, a \emph{descent} of $\sigma$ is an index $i\in\{1,\dotsc,n-1\}$ such that $\sigma(i)>\sigma(i+1)$. The number of descents of $\sigma$ is denoted by $\des(\sigma)$.} for $n\geq 1$ and $0\leq k\leq n-1$.
\end{defn}
Following~\cite{comtet} we present the Eulerian polynomial $A_n(q)$ as
\[
A_n(q)=\begin{cases}
{\displaystyle\sum_{k=1}^{n}A(n,k)q^k,}& n\geq 1,\\
1,&n=0.
\end{cases}
\]
According to the above formula one can compute the first Eulerian polynomials obtaining $A_1(q)=q$, $A_2(q)=q+q^2$, $A_3(q)=q+4q^2+q^3$.
The exponential generating function of the Eulerian polynomials is (see for instance~\cite[Section~6.5]{comtet}):
\begin{equation}\label{eq_eulerian_egf}
\sum_{n\geq 0}A_n(q)\frac{t^n}{n!}=\frac{1-q}{1-qe^{t(1-q)}}.
\end{equation}

The dimension of $H^{2k}(\XdeltaA{n-1})$ is $A(n,k+1)$ (see~\cite[Section 4]{stembridge92}), so the Poincaré polynomial of $\XdeltaA{n-1}$, written following the convention that $\deg q=2$, is
\[
P(\XdeltaA{n-1},q)=\sum_{k=0}^{n-1}A(n,k+1)q^{k}=\frac{1}{q}A_n(q).
\]

From Theorem~\ref{teo_main} we know that a basis for the cohomology of $\YYAT{n-1}\coloneqq\YY[\XdeltaA{n-1}]{\FFA{n-1}}$ is given by the elements of $\bb_{\FFA{n-1}}$. Recall that these elements are products of the form $b\,m_f$, where $f$ is an admissible function and $b\in \Theta(\supp f)$; we are going to study these two factors in this case.

\noindent\underline{Monomial $m_f$.} In analogy with~\cite{gaiffiselecta} and~\cite{yuzvinsky}, we can associate in a natural way an admissible function $f\colon\FFA{n-1}\to\NN$ with a so-called \emph{admissible forest} (on $n$ leaves).

\begin{defn}
An \emph{admissible tree} on $m$ leaves is a labeled directed rooted tree such that
\begin{itemize}
\item it has $m$ leaves, each labeled with a distinct non-zero natural number;
\item each non-leaf vertex $v$ has $k_v\geq 3$ outgoing edges, and it is labeled with the symbol $q^i$ where $i\in\{1,\dotsc,k_v-2\}$.
\end{itemize}
By convention, the graph with one vertex and no edges is an admissible tree on one leaf (actually the only one). The \emph{degree} of an admissible tree is the sum of the exponents of the labels of the non-leaf vertices.
\end{defn}

\begin{defn}
An \emph{admissible forest} on $n$ leaves is the disjoint union of admissible trees such that the sets of labels of their leaves form a partition of $\{1,\dotsc,n\}$. The \emph{degree} of an admissible forest is the sum of the degrees of its connected components.
\end{defn}


As illustrated by  Example~\ref{ex:monomieforeste}, the association between admissible forests and functions is the following: given an admissible forest $F$, for each internal vertex $v$ let $I(v)$ be the set of labels of the leaves that descend from $v$ and let $i(v)$ be such that $q^{i(v)}$ is the label of $v$; then the admissible function $f$ associated with $F$ has $\supp f=\{\kk_{I(v)}\st v\textrm{ internal vertex of }F\}$ and, for each $v$, $f\colon\kk_{I(v)}\mapsto i(v)$.

\begin{rmk}
The degree of an admissible forest associated with the function $f$ is equal to the degree of the monomial $m_f$.
\end{rmk}

\begin{exam}\label{ex:monomieforeste}
The admissible forest of Figure~\ref{fig_example_adm} is associated with the function $f$ with
\[
\supp f=\{\kk_{\{1,7,9,12\}},\kk_{\{8,10,13\}},\kk_{\{1,5,7,8,9,10,12,13\}},\kk_{\{2,6,11,14\}}\}
\]
and such that $f(\kk_{\{1,7,9,12\}})=2$, $f(\kk_{\{8,10,13\}})=1$, $f(\kk_{\{1,5,7,8,9,10,12,13\}})=1$, $f(\kk_{\{2,6,11,14\}})=1$. If we denote by $T_I$ the variable corresponding to $\kk_I\in \FFA{13}$, the monomial $m_f$ in this case is
\[
T_{\{1,7,9,12\}}^2 T_{\{8,10,13\}} T_{\{1,5,7,8,9,10,12,13\}} T_{\{2,6,11,14\}}.
\]

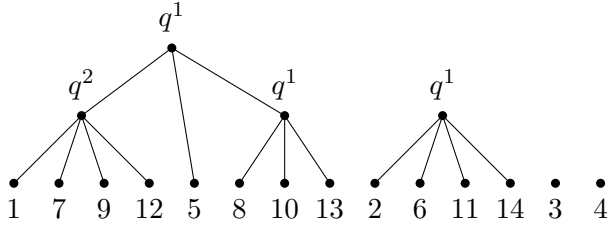
\begin{figure}[!htb]
\centering
\begin{tikzpicture}[scale=0.6]
\def\quota{1.5}
\draw (2.5,\quota) -- (1,0);
\draw (2.5,\quota) -- (2,0);
\draw (2.5,\quota) -- (3,0);
\draw (2.5,\quota) -- (4,0);
\draw (7,\quota) -- (6,0);
\draw (7,\quota) -- (7,0);
\draw (7,\quota) -- (8,0);
\draw (10.5,\quota) -- (9,0);
\draw (10.5,\quota) -- (10,0);
\draw (10.5,\quota) -- (11,0);
\draw (10.5,\quota) -- (12,0);
\draw (4.5,2*\quota) -- (2.5,\quota);
\draw (4.5,2*\quota) -- (5,0);
\draw (4.5,2*\quota) -- (7,\quota);
\fill (1,0) circle[radius=0.1] node[below=0.5ex] {$1$};
\fill (2,0) circle[radius=0.1] node[below=0.5ex] {$7$};
\fill (3,0) circle[radius=0.1] node[below=0.5ex] {$9$};
\fill (4,0) circle[radius=0.1] node[below=0.5ex] {$12$};
\fill (5,0) circle[radius=0.1] node[below=0.5ex] {$5$};
\fill (6,0) circle[radius=0.1] node[below=0.5ex] {$8$};
\fill (7,0) circle[radius=0.1] node[below=0.5ex] {$10$};
\fill (8,0) circle[radius=0.1] node[below=0.5ex] {$13$};
\fill (9,0) circle[radius=0.1] node[below=0.5ex] {$2$};
\fill (10,0) circle[radius=0.1] node[below=0.5ex] {$6$};
\fill (11,0) circle[radius=0.1] node[below=0.5ex] {$11$};
\fill (12,0) circle[radius=0.1] node[below=0.5ex] {$14$};
\fill (13,0) circle[radius=0.1] node[below=0.5ex] {$3$};
\fill (14,0) circle[radius=0.1] node[below=0.5ex] {$4$};
\fill (2.5,\quota) circle [radius=0.1] node[above=0.5ex] {$q^2$};
\fill (7,\quota) circle [radius=0.1] node[above=0.5ex] {$q^1$};
\fill (10.5,\quota) circle [radius=0.1] node[above=0.5ex] {$q^1$};
\fill (4.5,2*\quota) circle [radius=0.1] node[above=0.5ex] {$q^1$};
\end{tikzpicture}
\caption{An example of an admissible forest on 14 leaves with degree 5.}
\label{fig_example_adm}
\end{figure}
\end{exam}

\noindent\underline{Element $b\in\Theta(\supp f)$.} To study the elements $b\in\Theta(\supp f)$, first of all we need to analyze $H^*(X_{\DeltaA{n-1}(\supp f)})$: it is easy to show that the subfan $\DeltaA{n-1}(\supp f)$ is isomorphic to $\DeltaA{k-1}$, where $k$ is the number of connected components of the forest associated with $f$ (the isomorphism is obtained by identifying  the coordinates associated with the leaves of the same tree). Therefore the elements of $\Theta(\supp f)$ are in bijection with the permutations of $S_k$, and any statistics on $S_k$ that is equidistributed with the statistic $\des$ makes this bijection grade-preserving. We choose to use the so-called \emph{lec} statistic, first introduced in~\cite{foatahan}. To describe it, we need a couple of definitions. In the following, a permutation in $S_n$ will be denoted by the ordered $n$-tuple $[\sigma(1),\dotsc,\sigma(n)]$.

\subsection{Some remarks on the statistic lec}\label{sec:whylec}
Given an ordered list of distinct numbers (not necessarily a permutation), say $\sigma=[\sigma_1,\ldots,\sigma_n]$, we denote by $\inv(\sigma)$ the set of inversions of $\sigma$:
\[
\inv(\sigma):=\{(i,j)\mid i<j,\ \sigma_i>\sigma_j\}.
\]
\begin{defn}
A \emph{hook} is an ordered list of distinct non-zero natural numbers $\tau=[t_1,\dotsc,t_h]$, with $h\geq 2$, such that $t_1>t_2$ and $t_2<t_3<\dotsb<t_h$ (this second condition applies only for $h\geq 3$).
\end{defn}

\begin{rmk}\label{rmk:hook}
Given $s$ numbers $1\leq j_1<\dotsb<j_s\leq n$ and $i\in \{1,\ldots,s-1\}$
there is a unique way to sort $\{j_1,\ldots,j_s\}$ so that they form a hook with exactly $i$ inversions, namely $[j_{i+1},j_1,\ldots,j_i,j_{i+2},\ldots,j_s]$.
\end{rmk}

It is easy to observe that every list of distinct numbers has a unique \emph{hook factorization} (this notion comes from~\cite{gessel}), i.e.\ it is possible to write $\sigma$ as a concatenation $\sigma=p\tau_1\dotsm\tau_k$ where each $\tau_i$ is a hook and $p$ is a list of increasing numbers. Notice that it is possible to have $k=0$, if $\sigma$ is an increasing sequence; also it may happen that $p=\emptyset$ ($\sigma=[3,1,2]$ is an example with $k=1$). The statistic $\lec$ is defined as
\[
\lec(\sigma)=\sum_{i=1}^{k}\card{\inv(\tau_i)}
\]
where $p\tau_1\dotsb\tau_k$ is the hook factorization of $\sigma$.

\begin{exam}
Let $\sigma=[10,13,14,8,3,6,5,4,7,11,12,9,1,2]$. Its hook factorization is
\[
[10,13,14]\,[8,3,6]\,[5,4,7,11,12]\,[9,1,2]
\]
and $\lec(\sigma)=2+1+2=5$.
\end{exam}

Our choice of the statistic $\lec$ has been again inspired by  the theory of wonderful models. As it is well known (see for instance~\cite{batyrevblume}), $\XdeltaA{n-1}$ can be also seen as a projective wonderful model for the boolean hyperplane arrangement in $\PP^{n-1}$. More precisely, it is the maximal model: the (projective) hyperplanes are 
\[
H_i=\{[z_1,z_2,\dotsc,z_n]\in\PP^{n-1}\st z_i=0\}
\]
for \(i=1,\dotsc,n\) and the building set is provided by the full poset of their intersections.
The nested sets in this case are simply the chains of elements in this poset.

Therefore from~\cite{gaiffiselecta} we know how to describe  a monomial basis of  $H^*(\XdeltaA{n-1},\ZZ)$. In fact we will describe a  basis of the cohomology of the corresponding non-projective model, but the two cohomologies are isomorphic (this is a general property, see~\cite{wonderful1, gaiffiselecta}).

A monomial in this basis is a product of Chern classes associated with an admissible function (in analogy with our previous definitions; see also~\cite{yuzvinsky,gaiffiselecta}). In particular, the support of the (function associated with the) monomial is a chain of subsets of $\{1,\ldots,n\}$.

As an example let $n=10$ and consider the monomial
\[
\zeta_{\{1,2\}}\zeta_{\{1,2,4,5,6\}}^2\zeta_{\{1,2,4,5,6,7,8\}}
\]
which is an element of the basis of the cohomology of $\XdeltaA{9}$;
the variable $\zeta_{I}$, $I\subset\{1,\ldots, 10\}$, is the Chern class of the irreducible divisor obtained as proper transform of the subspace $H_{I}:=\cap_{i\in I}H_i$.
Notice that, for instance, the exponent of $\zeta_{\{1,2,4,5,6\}}$ is strictly less than $3$, i.e.\ the codimension of $H_{\{1,2,4,5,6\}}$ in $H_{\{1,2\}}$. 

We show an algorithm producing a bijection between this monomial  basis of $H^*(\XdeltaA{n-1},\ZZ)$ and \(S_n\), which is grade-preserving provided that we consider in \(S_n\) the grade induced by the statistic $\lec$.
The idea is to write a permutation $\sigma\in S_n$ in terms of its hook decomposition, associating a hook with every power of Chern class appearing in the monomial.
\begin{itemize}
\item We first look at the elements in $\{1,\ldots,n\}$ that do not appear in the support of the monomial. We write them in increasing order obtaining the non-hook part $p$ of $\sigma$. In our example we have $\{1,\dotsc,10\}\setminus \{1,2,4,5,6,7,8\}=\{3,9,10\}$ so $p=[3,9,10]$.
\item We then create the first hook of $\sigma$ by using Remark~\ref{rmk:hook} with the numbers in the smallest set of the support of the monomial, and the number of the inversions given by the corresponding exponent. In our example the smallest set is $\{1,2\}$ with exponent $1$ so $\tau_1=[2,1]$.
\item The second hook of $\sigma$ is formed using the numbers of the second set of the chain that do not appear in the smallest one. In the example those numbers are $\{4,5,6\}=\{1,2,4,5,6\}\setminus \{1,2\}$, so we form the hook with two inversions since $2$ is the exponent of $\zeta_{\{1,2,4,5,6\}}$ in the monomial: $\tau_2=[6,4,5]$. 
\item We go on building the $i$-th hook $\tau_i$ by looking at the numbers in the $i$-th set of the support of the monomial that do not appear in the $(i-1)$-th set. In our case there is only one set remaining: we pick $\{7,8\}$ from $\{1,2,4,5,6,7,8\}$ and form the hook $[8,7]$.
\end{itemize}
In the end we obtain $\sigma=[3,9,10]\,[2,1]\,[6,4,5]\,[8,7]$, which has $\lec(\sigma)=4$.

Notice that this bijection, if one already knows that the Betti numbers of $\XdeltaA{n-1}$ coincide with the Eulerian numbers, gives a geometric interpretation of the fact that the $\lec$ statistic is Eulerian.

\subsection{The toric model and the subspace model: an explicit bijection between their cohomology bases}
In Section~\ref{subsec:AnBasis} we have established a bijection between the basis $\bb_{\FFA{n-1}}$ and the set of pairs $(F,\sigma)$ where:
\begin{itemize}
\item $F$ is an admissible forest on $n$ leaves,
\item $\sigma$ is a permutation in $S_m$, where $m$ is the number of connected components of the forest $F$.
\end{itemize}
If we define the degree of a pair as $\deg(F,\sigma)=\deg(F)+\lec(\sigma)$, this bijection is grade-preserving.

Now, it can be proved that the model $\YYAT{n-1}$ is isomorphic to the projective model $\YYAH{n}$ for the hyperplane arrangement of type $A_n$, obtained by blowing up the building set of irreducible elements. Even if we don't use this fact in the present paper (we mention it only as inspiring additional information), we sketch here a proof.

The first step consists in noticing that both the toric and the hyperplane arrangements can be seen as the same subspace arrangement in a projective space of dimension $n-1$. On one side, the hyperplane arrangement of type $A_n$ can be seen as the arrangement in
\[
V=\{(0,x_1,\dotsc,x_n)\st x_i\in\CC\}\subseteq\CC^{n+1}
\]
given by the hyperplanes
\[
x_1=0,\dotsc,\ x_n=0,\ x_i-x_j=0\textrm{ for }1\leq i<j\leq n.
\]
The corresponding projective arrangement in $\PP(V)$ is given by the hyperplanes
\[
y_1=0,\dotsc,\ y_n=0,\ y_i-y_j=0\textrm{ for }1\leq i<j\leq n
\]
where, omitting the first zero, we denote by $[y_1,\dotsc,y_n]$ the projective coordinates of a point in $\PP(V)$. On the other side we observe that we can identify $(\CC^*)^{n-1}$ with
\[
\PP(\CC^n)\setminus\bigcup_{i=0}^{n-1}\{t_i=0\}
\]
via the map $(t_1,\dotsc,t_{n-1})\mapsto [1,t_1,\dotsc,t_{n-1}]$, where we denote by $[t_0,\dotsc,t_{n-1}]$ the projective coordinates in $\PP(\CC^n)$. In this setting, the divisorial layers of the toric arrangement of type $A_{n-1}$ $\TT=(\mathbb{C}^*)^n/H_n\simeq(\CC^*)^{n-1}$ are given by
\[
t_i=t_j\textrm{ for }0\leq i<j\leq n-1.
\]
Overall, we are considering in $\PP(\CC^n)$ the hyperplanes
\[
t_0=0,\dotsc,\ t_{n-1}=0,\ t_i-t_j=0\textrm{ for }0\leq i<j\leq n-1.
\]

The second step of the proof consists now in noticing that the two models $\YYAT{n-1}$ and $\YYAH{n}$ are obtained by blowing up the same subspaces; however, the two constructions differ in the order in which the blow ups are carried out, but thanks to~\cite[Theorem~1.3]{li} the two resulting varieties are isomorphic.

This suggests us to search for a grade-preserving bijection between the bases of the cohomologies of $\YYAT{n-1}$ and $\YYAH{n}$. Recall that a basis for the cohomology of $\YYAH{n}$ is in grade-preserving bijection with the set of admissible forests on $n+1$ leaves~\cite{yuzvinsky,gaiffiselecta}. So we describe an algorithm that produces an explicit bijection $\Psi$, associating a pair $(F,\sigma)$ (given by  an admissible forest $F$ on $n$ leaves with  $m$ trees and a permutation $\sigma\in S_m$) with an admissible forest $\mathcal{F}$ on $n+1$ leaves.

As a preliminary step we fix an ordering of the trees in $F$. For example we can say that $T<T'$ if the minimum index labelling the leaves in $T$ is smaller than the minimum index labelling the leaves in $T'$. We denote the trees accordingly as $T_1<T_2<\dotsb<T_m$.

Let $\sigma=p\tau_1\dotsm\tau_k$ be the hook factorization of $\sigma$.

\noindent\underline{Base step: $k=0$.} In this case $\mathcal{F}$ is obtained from $F$ by simply adding a connected component with a single vertex-leaf labeled with $n+1$.

\noindent\underline{Inductive step: $k>0$.} Let $\tau_k=[M,a_1,\dotsc,a_\ell]$ be the last hook of $\sigma$ and let $i=\card{\inv(\tau_k)}$, i.e.\ the number of inversions. We produce a tree $T$ connecting with a new internal labeled vertex the roots of the trees $T_M,T_{a_1},\ldots,T_{a_\ell}$ and an extra leaf labeled with $n+1$. We label this new internal vertex, which is the root of $T$, with $q^i$. Figure~\ref{fig_bijection_tree} shows the situation.
\begin{figure}[!htb]
\centering
\begin{tikzpicture}[scale=0.6]
\draw[line width=2pt,rounded corners,gray!25] (-1,-2.2) rectangle (1,0.5);
\draw (0,0) -- (-0.6,-1);
\draw (0,0) -- (0,-0.8);
\draw (0,0) -- (0.6,-1);
\fill (0,0) circle[radius=0.1];
\draw (0,-0.8) node[below]{$\ldots$};
\draw (0,-1.7) node {$T_{a_1}$};

\draw[line width=2pt,rounded corners,gray!25] (-3.5,-2.2) rectangle (-1.5,0.5);
\draw (-2.5,0) -- (-3.1,-1);
\draw (-2.5,0) -- (-2.5,-0.8);
\draw (-2.5,0) -- (-1.9,-1);
\fill (-2.5,0) circle[radius=0.1];
\draw (-2.5,-0.8) node[below]{$\ldots$};
\draw (-2.5,-1.7) node {$T_{M}$};

\draw[line width=2pt,rounded corners,gray!25] (2.5,-2.2) rectangle (4.5,0.5);
\draw (3.5,0) -- (4.1,-1); 
\draw (3.5,0) -- (3.5,-0.8);
\draw (3.5,0) -- (2.9,-1);
\fill (3.5,0) circle[radius=0.1];
\draw (3.5,-0.8) node[below]{$\ldots$};
\draw (3.5,-1.7) node {$T_{a_\ell}$};

\fill (7,-1) circle[radius=0.1] node[below] {$n+1$};

\draw (2.5,2.5) -- (-2.5,0);
\draw (2.5,2.5) -- (0,0);
\draw (2.5,2.5) -- (3.5,0);
\draw (2.5,2.5) -- (7,-1);

\node at (1.75,-0.6) {$\ldots$};
\fill (2.5,2.5) circle [radius=0.1] node[above right] {$q^i$};
\end{tikzpicture}
\caption{The new tree $T$ obtained using the last hook $\tau_k$ of $\sigma$.}
\label{fig_bijection_tree}
\end{figure}
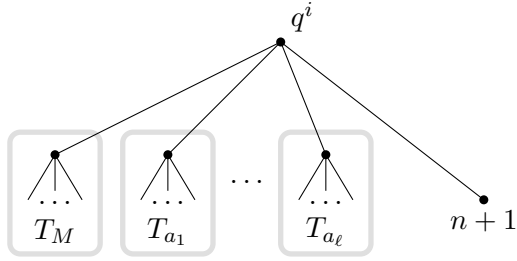

Now we consider the forest $F'$ obtained from $F$ by removing the trees $T_M,T_{a_1},\ldots,T_{a_\ell}$ and the list $\sigma'=p\tau_1\dotsb \tau_{k-1}$ and we apply the same construction to the pair $(F',\sigma')$, with the difference that now, instead of connecting the trees to an extra leaf labeled $n+1$, we connect them to the root of the tree $T$ obtained in the previous step. The algorithm is repeated inductively until there are no hooks remaining.

To prove that this algorithm defines a bijection, we present the reverse algorithm that computes $\Psi^{-1}$, associating an admissible forest $\mathcal{F}$ on $n+1$ leaves with a pair $(F,\sigma)$.

\noindent\underline{Description of the forest $F$.} As the first step of the reverse algorithm, we remove all the internal vertices of $\mathcal{F}$ that have the leaf labeled $n+1$ among their descendants. When we remove a vertex, we remove also its label and all its outgoing edges (but not their descendants). Then we remove the leaf labeled with $n+1$. In this way, we have obtained a forest $F$ on $n$ leaves, and we sort its connected components according to the usual ordering $T_1,\ldots,T_m$ (see Figure~\ref{fig_reverse_bijection}).
\begin{figure}[!htb]
\centering
\begin{tikzpicture}
\def\quota{1.2}
\node[anchor=south] at (0,0) {%
\begin{tikzpicture}[scale=0.6]
\draw (2,\quota) -- (1,0);
\draw (2,\quota) -- (2,0);
\draw (2,\quota) -- (3,0);
\draw (3.5,2*\quota) -- (2,\quota);
\draw (3.5,2*\quota) -- (4,0);
\draw (3.5,2*\quota) -- (5,0);
\draw (3.5,2*\quota) -- (6,0);
\fill (1,0) circle[radius=0.1] node[below=0.5ex] {$1$};
\fill (2,0) circle[radius=0.1] node[below=0.5ex] {$4$};
\fill (3,0) circle[radius=0.1] node[below=0.5ex] {$5$};
\fill (4,0) circle[radius=0.1] node[below=0.5ex] {$2$};
\fill (5,0) circle[radius=0.1] node[below=0.5ex] {$3$};
\fill (6,0) circle[radius=0.1] node[below=0.5ex] {$7$};
\fill (7,0) circle[radius=0.1] node[below=0.5ex] {$6$};
\fill (2,\quota) circle [radius=0.1] node[above=0.5ex] {$q^1$};
\fill (3.5,2*\quota) circle [radius=0.1] node[above=0.5ex] {$q^1$};
\end{tikzpicture}};
\node at (3,\quota) {$\longrightarrow$};
\node[anchor=south] at (6,0) {%
\begin{tikzpicture}[scale=0.6]
\draw (2,\quota) -- (1,0);
\draw (2,\quota) -- (2,0);
\draw (2,\quota) -- (3,0);
\fill (1,0) circle[radius=0.1] node[below=0.5ex] {$1$};
\fill (2,0) circle[radius=0.1] node[below=0.5ex] {$4$};
\fill (3,0) circle[radius=0.1] node[below=0.5ex] {$5$};
\fill (4,0) circle[radius=0.1] node[below=0.5ex] {$2$};
\fill (5,0) circle[radius=0.1] node[below=0.5ex] {$3$};
\fill (6,0) circle[radius=0.1] node[below=0.5ex] {$6$};
\fill (2,\quota) circle [radius=0.1] node[above=0.5ex] {$q^1$};
\end{tikzpicture}};
\end{tikzpicture}
\caption{An example of the application of the reverse algorithm.}
\label{fig_reverse_bijection}
\end{figure}
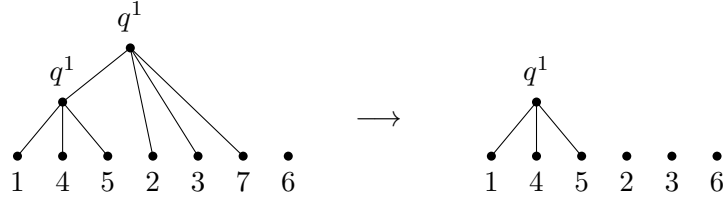

\noindent\underline{Description of the permutation $\sigma$.} We now need to describe $\sigma\in S_m$. If $\mathcal{F}$ has no internal vertices with the leaf $n+1$ as a descendant, we just take $\sigma=e$, the identity in $S_m$. Otherwise, let $v$ be the vertex in $\mathcal{F}$ that covers $n+1$; let $q^i$ be its label and let $\{T_{a_1},\ldots,T_{a_s},n+1\}$ be the set of connected components of the forest obtained by removing $v$ and its outgoing edges from the subtree of $\mathcal{F}$ with root $v$. The situation is described in Figure~\ref{fig_bijection_vertex_v}. Since $\mathcal{F}$ is an admissible forest we have $1\leq i\leq s-1$, so we can apply Remark~\ref{rmk:hook} to the set $\{a_1,\ldots,a_s\}$  and obtain the hook $[a_{i+1},a_1,\ldots,a_{i},a_{i+2},\ldots,a_s]$, which will be the last hook of the permutation $\sigma$.
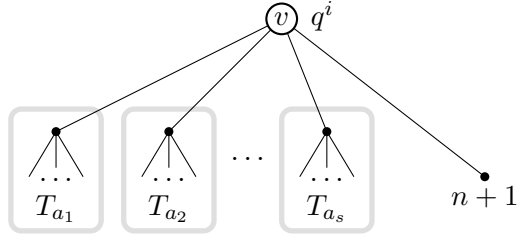
\begin{figure}[!htb]
\centering
\usetikzlibrary{calc}
\begin{tikzpicture}[scale=0.6]
\draw[line width=2pt,rounded corners,gray!25] (-1,-2.2) rectangle (1,0.5);
\draw (0,0) -- (-0.6,-1);
\draw (0,0) -- (0,-0.8);
\draw (0,0) -- (0.6,-1);
\fill (0,0) circle[radius=0.1];
\draw (0,-0.8) node[below]{$\ldots$};
\draw (0,-1.7) node {$T_{a_2}$};

\draw[line width=2pt,rounded corners,gray!25] (-3.5,-2.2) rectangle (-1.5,0.5);
\draw (-2.5,0) -- (-3.1,-1);
\draw (-2.5,0) -- (-2.5,-0.8);
\draw (-2.5,0) -- (-1.9,-1);
\fill (-2.5,0) circle[radius=0.1];
\draw (-2.5,-0.8) node[below]{$\ldots$};
\draw (-2.5,-1.7) node {$T_{a_1}$};

\draw[line width=2pt,rounded corners,gray!25] (2.5,-2.2) rectangle (4.5,0.5);
\draw (3.5,0) -- (4.1,-1); 
\draw (3.5,0) -- (3.5,-0.8);
\draw (3.5,0) -- (2.9,-1);
\fill (3.5,0) circle[radius=0.1];
\draw (3.5,-0.8) node[below]{$\ldots$};
\draw (3.5,-1.7) node {$T_{a_s}$};

\fill (7,-1) circle[radius=0.1] node[below] {$n+1$};

\draw (2.5,2.5) -- (-2.5,0);
\draw (2.5,2.5) -- (0,0);
\draw (2.5,2.5) -- (3.5,0);
\draw (2.5,2.5) -- (7,-1);

\node at (1.75,-0.6) {$\ldots$};
\node[circle,fill=white,draw,thick,inner sep=1.5pt] (v) at (2.5,2.5) {$v$};
\node[inner sep=0pt] at (3.4,2.57) {$q^i$};
\end{tikzpicture}
\caption{The vertex $v$ covers the leaf labeled with $n+1$.}
\label{fig_bijection_vertex_v}
\end{figure}

Let now $\{1,\ldots,m\}\setminus \{a_1,\ldots,a_s\}=\{b_{1},\ldots,b_{m-s}\}$, with $b_1<\dotsb<b_{m-s}$. If in $\mathcal{F}$ there are no vertices that cover $v$ we define
\[
\sigma=[b_1,\ldots,b_{m-s},a_{i+1},a_1,\ldots,a_{i},a_{i+2},\ldots,a_s];
\]
if instead there is a vertex, say $w$, that covers $v$ in $\mathcal{F}$ we have a picture like Figure~\ref{fig_bijection_vertex_w}, with $c_1<\dotsb<c_h$ and $1\leq r \leq h-1$. We repeat the same step as we did for $v$, obtaining a new hook $[c_{r+1},c_1,\ldots,c_r,c_{r+2},\ldots,c_h]$ so that the last part of $\sigma$ is now
\[
[c_{r+1},c_1,\ldots,c_r,c_{r+2},\ldots,c_h,a_{i+1},a_1,\ldots,a_{i-1},a_{i+2},\ldots,a_s].
\]
\begin{figure}[!htb]
\centering
\usetikzlibrary{calc}
\begin{tikzpicture}[scale=0.6]
\draw[line width=2pt,rounded corners,gray!25] (-1,-2.2) rectangle (1,0.5);
\draw (0,0) -- (-0.6,-1);
\draw (0,0) -- (0,-0.8);
\draw (0,0) -- (0.6,-1);
\fill (0,0) circle[radius=0.1];
\draw (0,-0.8) node[below]{$\ldots$};
\draw (0,-1.7) node {$T_{c_1}$};

\draw[line width=2pt,rounded corners,gray!25] (2.5,-2.2) rectangle (4.5,0.5);
\draw (3.5,0) -- (4.1,-1); 
\draw (3.5,0) -- (3.5,-0.8);
\draw (3.5,0) -- (2.9,-1);
\fill (3.5,0) circle[radius=0.1];
\draw (3.5,-0.8) node[below]{$\ldots$};
\draw (3.5,-1.7) node {$T_{c_h}$};

\draw[line width=2pt,rounded corners,gray!25] (5.5,-2.2) rectangle (8.5,0.5);
\draw (7,0) -- (6.4,-1);
\draw (7,0) -- (7,-0.8);
\draw (7,0) -- (7.6,-1);
\draw (7,-0.8) node[below]{$\ldots$};
\draw (7,-1.7) node {(Fig.~\ref{fig_bijection_vertex_v})};

\draw (3,2.5) -- (3.5,0);
\draw (3,2.5) -- (0,0);
\draw (3,2.5) -- (7,0);

\node[circle,fill=white,draw,thick,inner sep=1pt] (v) at (7,0) {$v$};

\node at (1.75,-0.6) {$\ldots$};
\node[circle,fill=white,draw,thick,inner sep=1.5pt] (v) at (3,2.5) {$w$};
\node[inner sep=0pt] at (3.9,2.57) {$q^r$};
\end{tikzpicture}
\caption{The vertex $w$ covers the vertex $v$ of Figure~\ref{fig_bijection_vertex_v}.}
\label{fig_bijection_vertex_w}
\end{figure}
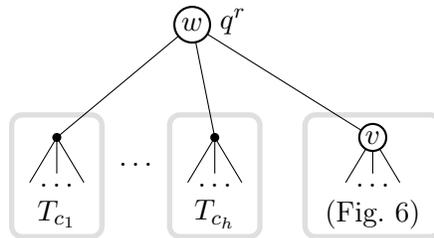

We repeat the previous steps as long as there are internal vertices in $\mathcal{F}$ covering the last vertex that we removed.
\subsection{A combinatorial proof, with a geometrical interpretation, that lec is Eulerian}
The bijection $\Psi$ described above allows us to give a new proof that $\lec$ is an Eulerian statistic. This proof is purely combinatorial, and therefore in particular it differs from the one sketched in Section~\ref{sec:whylec}, which uses the fact that the Betti numbers of $\XdeltaA{n-1}$ are Eulerian numbers.  Nevertheless our proof has a geometric inspiration that comes from counting elements of monomial bases of cohomologies of models. We first need to introduce some generating functions.

Let $\lambda(q,t)$ be the generating function of the admissible trees, i.e.\ the series whose coefficient of ${q^it^k}/{k!}$ counts the number of admissible trees of degree $i$ on $k$ leaves (see~\cite{yuzvinsky,gaiffiselecta,manin}). There are explicit combinatorial ways to compute the series $\lambda$, as the following theorem shows. 
\begin{thm}[{\cite[Theorem 4.1]{gaiffiselecta}}]\label{teo_recurrence_lambda}
Let $\lambda$ defined as above. Then we have the following recurrence relation:
\[
\frac{\partial}{\partial t}\lambda=1+\frac{\frac{\partial}{\partial t}\lambda}{q-1}(e^{q\lambda}-qe^{\lambda}+q-1).
\]
In other words
\begin{equation}\label{eq_lambda_primo}
\frac{\partial}{\partial t}\lambda = \frac{1-q}{e^{q\lambda}-qe^{\lambda}}.
\end{equation}
\end{thm}
The first few terms of $\lambda$ are
\[
\lambda(q,t)=t+q\frac{t^3}{3!}+(q+q^2)\frac{t^4}{4!}+\dotsb
\]
By standard combinatorial arguments we deduce that the generating function of the admissible forests is $e^\lambda-1$, and in particular the number of the admissible forests with $k$ connected components on $n$ leaves and degree $d$ is counted by the coefficient of $q^d{t^n}/{n!}$ in the series ${\lambda^k(q,t)}/{k!}$.

We define now the exponential generating function for the $\lec$ statistic
\[
\mathcal{L}(q,t)\coloneqq\sum_{n\geq 1}\left(\sum_{\sigma\in S_n}q^{\lec(\sigma)} \right)\frac{t^n}{n!}
\]
and the usual exponential generating function for any Eulerian statistic
\[
\mathcal{E}(q,t)\coloneqq\sum_{n\geq 1}\left(\sum_{\sigma\in S_n}q^{\des(\sigma)} \right)\frac{t^n}{n!}=\sum_{n\geq 1}\frac{A_n(q)}{q}\frac{t^n}{n!}.
\]
Our goal is to prove that $\mathcal{L}(q,t)= \mathcal{E}(q,t)$. 
This is equivalent to prove that   
\[
\mathcal{L}(q,\lambda(q,t))=\mathcal{E}(q,\lambda(q,t))
\]
since  $\lambda$, viewed as a series in $\ZZ[q][[t]]$, is invertible with respect to the composition (its constant term is zero and its degree~1 term is invertible in $\ZZ[q]$).

Notice that the coefficient of $q^dt^n/n!$ in the series $\mathcal{L}(q,\lambda(q,t))$ counts the pairs $(F,\sigma)$ where $F$ is an admissible forest on $n$ leaves and $\deg(F,\sigma)=d$; from the bijection $\Psi$ we deduce that the series $\mathcal{L}(q,\lambda(q,t))$ is equal to the series $e^\lambda-1$ shifted by one, i.e.
\begin{equation}\label{eq_PhiH_bijection}
\mathcal{L}(q,\lambda(q,t))=\frac{\partial}{\partial t}(e^{\lambda(q,t)}-1)-1.
\end{equation}
Now, a simple computation of formal power series gives that
\begin{equation}\label{eq_PhiT_PhiH}
\frac{\partial}{\partial t}(e^{\lambda(q,t)}-1)-1=\mathcal{E}(q,\lambda(q,t)).
\end{equation}
In fact from~\eqref{eq_eulerian_egf} we can write
\[
q\cdot\mathcal{E}(q,\lambda(q,t))=\frac{1-q}{1-qe^{\lambda(q,t)\cdot(1-q)}}-1
\]
from which we have
\[
\mathcal{E}(q,\lambda(q,t))=\frac{-1+e^{\lambda(q,t)\cdot (1-q)}}{1-qe^{\lambda(q,t)\cdot(1-q)}}=\frac{e^{\lambda(q,t)}-e^{\lambda(q,t)\cdot q}}{e^{\lambda(q,t)\cdot q}-qe^{\lambda(q,t)}};
\]
on the other hand
\begin{align*}
\frac{\partial}{\partial t}(e^{\lambda(q,t)}-1) -1&{}\stackrel{\phantom{\eqref{eq_lambda_primo}}}{=}e^{\lambda(q,t)}\frac{\partial}{\partial t}(\lambda(q,t))-1=\\
&{}\stackrel{\eqref{eq_lambda_primo}}{=}\frac{e^{\lambda(q,t)}(1-q)}{e^{\lambda(q,t)\cdot q}-qe^{\lambda(q,t)}}-1=
\frac{e^{\lambda(q,t)}-e^{\lambda(q,t)\cdot q}}{e^{\lambda(q,t)\cdot q}-qe^{\lambda(q,t)}}.
\end{align*}
By combining~\eqref{eq_PhiH_bijection} and~\eqref{eq_PhiT_PhiH} we conclude.

This proof has used only combinatorial arguments, but, as we remarked above, it has a geometric inspiration. We notice that some of the power series involved are actually the generating functions for the Poincaré polynomials of the compact models $\YYAT{n-1}$ and $\YYAH{n-1}$, defined as
\begin{align*}
\Phi^T(q,t)&{}:=\sum_{n\geq 1}P(\YYAT{n-1},q)\frac{t^n}{n!},\\
\Phi^H(q,t)&{}:=\sum_{n\geq 1}P(\YYAH{n-1},q)\frac{t^n}{n!}.
\end{align*}
In fact we already know that $\Phi^H(q,t)=e^{\lambda(q,t)}-1$, and the description of the basis for the toric model gives that $\Phi^T(q,t)=\mathcal{E}(q,\lambda(q,t))$. From this point of view we can read Equation~\eqref{eq_PhiT_PhiH} as
\[
\frac{\partial}{\partial t}\Phi^H(q,t) -1=\Phi^T(q,t),
\]
which reveals itself to be a consequence of the isomorphism between $\YYAT{n-1}$ and $\YYAH{n}$.

\appendix
\section{Description of the fan in Example~\ref{ex_arrangement}}\label{app:examplefan}
The following table lists the rays of the fan $\Delta$ associated with a good toric variety for the arrangement of Example~\ref{ex_arrangement}.
\begin{center}
\small
\begin{longtable}{>{$}c<{$}>{$}c<{$}>{$}c<{$}>{$}c<{$}}
\toprule
r_{1}\colon \left(0, -2, -1\right) &
r_{2}\colon \left(0, -1, -1\right) &
r_{3}\colon \left(-2, 1, -1\right) &
r_{4}\colon \left(-1, 1, -1\right) \\
r_{5}\colon \left(0, -1, 1\right) &
r_{6}\colon \left(0, 1, 0\right) &
r_{7}\colon \left(-2, 1, 1\right) &
r_{8}\colon \left(-1, 1, 1\right) \\
r_{9}\colon \left(1, -1, 1\right) &
r_{10}\colon \left(-1, 3, 1\right) &
r_{11}\colon \left(-2, 3, 1\right) &
r_{12}\colon \left(1, -2, -1\right) \\
r_{13}\colon \left(1, -1, -1\right) &
r_{14}\colon \left(6, -3, -1\right) &
r_{15}\colon \left(2, -1, -1\right) &
r_{16}\colon \left(2, -2, -1\right) \\
r_{17}\colon \left(2, -1, -2\right) &
r_{18}\colon \left(0, 0, -1\right) &
r_{19}\colon \left(2, -1, 1\right) &
r_{20}\colon \left(0, 0, 1\right) \\
r_{21}\colon \left(0, 2, 1\right) &
r_{22}\colon \left(1, 0, 1\right) &
r_{23}\colon \left(1, 0, -1\right) &
r_{24}\colon \left(-6, 3, 1\right) \\
r_{25}\colon \left(-2, -2, -1\right) &
r_{26}\colon \left(-2, -1, -1\right) &
r_{27}\colon \left(4, -3, -1\right) &
r_{28}\colon \left(0, -1, 0\right) \\
r_{29}\colon \left(2, 0, -1\right) &
r_{30}\colon \left(-5, 3, 1\right) &
r_{31}\colon \left(-2, -1, 1\right) &
r_{32}\colon \left(2, 0, 1\right) \\
r_{33}\colon \left(-1, -1, 1\right) &
r_{34}\colon \left(-1, -2, -1\right) &
r_{35}\colon \left(-1, 1, 0\right) &
r_{36}\colon \left(-2, 1, 0\right) \\
r_{37}\colon \left(2, 2, 1\right) &
r_{38}\colon \left(3, -3, -1\right) &
r_{39}\colon \left(1, -1, 0\right) &
r_{40}\colon \left(-4, 2, 1\right) \\
r_{41}\colon \left(-1, -1, -1\right) &
r_{42}\colon \left(2, -1, 0\right) &
r_{43}\colon \left(1, 2, 1\right) &
r_{44}\colon \left(5, -3, -1\right) \\
r_{45}\colon \left(-3, 3, 1\right) &
r_{46}\colon \left(-2, 0, 1\right) &
r_{47}\colon \left(-1, 0, 1\right) &
r_{48}\colon \left(3, -2, -1\right) \\
r_{49}\colon \left(-2, 0, -1\right) &
r_{50}\colon \left(-1, 0, -1\right) &
r_{51}\colon \left(-1, 2, 1\right) &
r_{52}\colon \left(-2, 2, 1\right) \\
r_{53}\colon \left(1, 0, 0\right) &
r_{54}\colon \left(0, -3, -1\right) &
r_{55}\colon \left(2, -3, -1\right) &
r_{56}\colon \left(-4, 3, 1\right) \\
r_{57}\colon \left(0, 1, -1\right) &
r_{58}\colon \left(2, 1, -1\right) &
r_{59}\colon \left(1, -3, -1\right) &
r_{60}\colon \left(0, 1, 1\right) \\
r_{61}\colon \left(-2, 1, 2\right) &
r_{62}\colon \left(2, 1, 1\right) &
r_{63}\colon \left(0, 3, 1\right) &
r_{64}\colon \left(-2, -3, -1\right) \\
r_{65}\colon \left(-1, -3, -1\right) &
r_{66}\colon \left(1, 3, 1\right) &
r_{67}\colon \left(4, -2, -1\right) &
r_{68}\colon \left(-1, 0, 0\right) \\
r_{69}\colon \left(2, 3, 1\right) &
r_{70}\colon \left(1, 1, -1\right) &
r_{71}\colon \left(1, 1, 1\right) &
r_{72}\colon \left(-3, 2, 1\right) \\
\bottomrule
\end{longtable}
\end{center}

The following table lists the maximal cones of the fan $\Delta$ associated with a good toric variety for the arrangement of Example~\ref{ex_arrangement}. Each cone is given by its generating rays.
\begin{center}
\begin{longtable}{>{$}c<{$}>{$}c<{$}>{$}c<{$}>{$}c<{$}}
\toprule
C(r_{6},r_{53},r_{69}) &
C(r_{37},r_{53},r_{69}) &
C(r_{37},r_{53},r_{62}) &
C(r_{6},r_{66},r_{69}) \\
C(r_{37},r_{66},r_{69}) &
C(r_{37},r_{43},r_{66}) &
C(r_{37},r_{43},r_{62}) &
C(r_{43},r_{62},r_{71}) \\
C(r_{6},r_{63},r_{66}) &
C(r_{21},r_{63},r_{66}) &
C(r_{21},r_{43},r_{66}) &
C(r_{21},r_{43},r_{60}) \\
C(r_{43},r_{60},r_{71}) &
C(r_{32},r_{53},r_{62}) &
C(r_{22},r_{32},r_{62}) &
C(r_{22},r_{62},r_{71}) \\
C(r_{22},r_{60},r_{71}) &
C(r_{20},r_{22},r_{60}) &
C(r_{6},r_{53},r_{58}) &
C(r_{6},r_{58},r_{70}) \\
C(r_{6},r_{57},r_{70}) &
C(r_{29},r_{53},r_{58}) &
C(r_{23},r_{29},r_{58}) &
C(r_{23},r_{58},r_{70}) \\
C(r_{18},r_{23},r_{70}) &
C(r_{18},r_{57},r_{70}) &
C(r_{19},r_{28},r_{39}) &
C(r_{19},r_{42},r_{53}) \\
C(r_{19},r_{39},r_{42}) &
C(r_{9},r_{19},r_{28}) &
C(r_{5},r_{9},r_{28}) &
C(r_{19},r_{32},r_{53}) \\
C(r_{19},r_{22},r_{32}) &
C(r_{19},r_{20},r_{22}) &
C(r_{9},r_{19},r_{20}) &
C(r_{5},r_{9},r_{20}) \\
C(r_{28},r_{54},r_{59}) &
C(r_{1},r_{54},r_{59}) &
C(r_{1},r_{12},r_{59}) &
C(r_{1},r_{2},r_{12}) \\
C(r_{28},r_{39},r_{55}) &
C(r_{28},r_{55},r_{59}) &
C(r_{12},r_{55},r_{59}) &
C(r_{15},r_{53},r_{67}) \\
C(r_{14},r_{42},r_{53}) &
C(r_{14},r_{53},r_{67}) &
C(r_{27},r_{39},r_{42}) &
C(r_{27},r_{42},r_{44}) \\
C(r_{27},r_{44},r_{48}) &
C(r_{15},r_{48},r_{67}) &
C(r_{14},r_{42},r_{44}) &
C(r_{14},r_{44},r_{48}) \\
C(r_{14},r_{48},r_{67}) &
C(r_{27},r_{38},r_{39}) &
C(r_{27},r_{38},r_{48}) &
C(r_{16},r_{38},r_{48}) \\
C(r_{15},r_{16},r_{48}) &
C(r_{13},r_{15},r_{16}) &
C(r_{2},r_{12},r_{13}) &
C(r_{38},r_{39},r_{55}) \\
C(r_{16},r_{38},r_{55}) &
C(r_{13},r_{16},r_{55}) &
C(r_{12},r_{13},r_{55}) &
C(r_{15},r_{29},r_{53}) \\
C(r_{15},r_{23},r_{29}) &
C(r_{13},r_{15},r_{17}) &
C(r_{15},r_{17},r_{23}) &
C(r_{13},r_{17},r_{18}) \\
C(r_{17},r_{18},r_{23}) &
C(r_{2},r_{13},r_{18}) &
C(r_{6},r_{10},r_{63}) &
C(r_{10},r_{21},r_{63}) \\
C(r_{10},r_{21},r_{51}) &
C(r_{21},r_{51},r_{60}) &
C(r_{6},r_{11},r_{35}) &
C(r_{6},r_{10},r_{11}) \\
C(r_{10},r_{11},r_{51}) &
C(r_{7},r_{40},r_{68}) &
C(r_{24},r_{36},r_{68}) &
C(r_{24},r_{40},r_{68}) \\
C(r_{35},r_{36},r_{56}) &
C(r_{30},r_{36},r_{56}) &
C(r_{30},r_{56},r_{72}) &
C(r_{7},r_{40},r_{72}) \\
C(r_{24},r_{30},r_{36}) &
C(r_{24},r_{30},r_{72}) &
C(r_{24},r_{40},r_{72}) &
C(r_{35},r_{45},r_{56}) \\
C(r_{45},r_{56},r_{72}) &
C(r_{45},r_{52},r_{72}) &
C(r_{7},r_{52},r_{72}) &
C(r_{7},r_{8},r_{52}) \\
C(r_{8},r_{51},r_{60}) &
C(r_{11},r_{35},r_{45}) &
C(r_{11},r_{45},r_{52}) &
C(r_{8},r_{11},r_{52}) \\
C(r_{8},r_{11},r_{51}) &
C(r_{7},r_{46},r_{68}) &
C(r_{7},r_{46},r_{47}) &
C(r_{7},r_{8},r_{61}) \\
C(r_{7},r_{47},r_{61}) &
C(r_{8},r_{20},r_{61}) &
C(r_{20},r_{47},r_{61}) &
C(r_{8},r_{20},r_{60}) \\
C(r_{3},r_{6},r_{35}) &
C(r_{3},r_{36},r_{68}) &
C(r_{3},r_{35},r_{36}) &
C(r_{3},r_{4},r_{6}) \\
C(r_{4},r_{6},r_{57}) &
C(r_{3},r_{49},r_{68}) &
C(r_{3},r_{49},r_{50}) &
C(r_{3},r_{18},r_{50}) \\
C(r_{3},r_{4},r_{18}) &
C(r_{4},r_{18},r_{57}) &
C(r_{28},r_{31},r_{68}) &
C(r_{28},r_{31},r_{33}) \\
C(r_{5},r_{28},r_{33}) &
C(r_{31},r_{46},r_{68}) &
C(r_{31},r_{46},r_{47}) &
C(r_{31},r_{33},r_{47}) \\
C(r_{20},r_{33},r_{47}) &
C(r_{5},r_{20},r_{33}) &
C(r_{28},r_{64},r_{68}) &
C(r_{25},r_{64},r_{68}) \\
C(r_{25},r_{26},r_{68}) &
C(r_{28},r_{64},r_{65}) &
C(r_{25},r_{64},r_{65}) &
C(r_{25},r_{34},r_{65}) \\
C(r_{25},r_{26},r_{34}) &
C(r_{26},r_{34},r_{41}) &
C(r_{28},r_{54},r_{65}) &
C(r_{1},r_{54},r_{65}) \\
C(r_{1},r_{34},r_{65}) &
C(r_{1},r_{2},r_{34}) &
C(r_{2},r_{34},r_{41}) &
C(r_{26},r_{49},r_{68}) \\
C(r_{26},r_{49},r_{50}) &
C(r_{26},r_{41},r_{50}) &
C(r_{2},r_{41},r_{50}) &
C(r_{2},r_{18},r_{50}) \\
\bottomrule
\end{longtable}
\end{center}

\section*{Acknowledgments}
The authors would like to thank Michele D'Adderio for the useful conversations and for pointing out the statistic $\lec$ introduced in~\cite{foatahan}. The authors also acknowledge the support of INdAM-GNSAGA.

\printbibliography
\end{document}